\documentclass{amsart}
\usepackage{amsfonts,amssymb,stmaryrd,amscd,amsmath,latexsym,amsbsy,wasysym}

\numberwithin{equation}{section}
\theoremstyle{plain}
\newtheorem{thm}{Theorem}[section]

\newtheorem{prop}[thm]{Proposition}

\newtheorem{defi}[thm]{Definition}
\newtheorem{lem}[thm]{Lemma}
\newtheorem{cor}[thm]{Corollary}
\newtheorem{conj}[thm]{Conjecture}
\newtheorem{eg}[thm]{{Example}}

\theoremstyle{remark}
\newtheorem{rema}[thm]{Remark}

\newcommand{\abv}[2]{\genfrac{}{}{0pt}{}{#1}{#2}}
\newcommand{\ad}{{\mbox{\upshape{ad}}}}
\newcommand{\Ad}{{\mbox{\upshape{Ad}}}}

\newcommand{\Aut}{\mathrm{Aut}}

\newcommand{\bc}{{\mathbf{c}}}
\newcommand{\bd}{{\mathbf{d}}}
\newcommand{\bs}{{\mathbf{s}}}

\newcommand{\C}{{\mathbb C}}

\newcommand{\N}{{\mathbb N}}

\newcommand{\cC}{{\mathcal C}}

\newcommand{\cD}{{\mathcal D}}

\newcommand{\cH}{{\mathcal H}}

\newcommand{\cJ}{{\mathcal J}}

\newcommand{\cM}{{\mathcal M}}
\newcommand{\cO}{{\mathcal O}}

\newcommand{\cS}{{\mathcal S}}
\newcommand{\cU}{{\mathcal U}}
\newcommand{\cW}{{\mathcal W}}
\newcommand{\cZ}{{\mathcal Z}}

\newcommand{\field}{{\mathbb K}}

\newcommand{\gfrak}{{\mathfrak g}}
\newcommand{\glfrak}{{\mathfrak{gl}}}

\newcommand{\hfrak}{{\mathfrak h}}

\newcommand{\Htil}{\tilde{H}}
\newcommand{\Hom}{{\mathrm{Hom}}}

\newcommand{\id}{{\mbox{id}}}

\newcommand{\kfrak}{{\mathfrak k}}
\newcommand{\kow}{{\varDelta}}

\newcommand{\ot}{\otimes}

\newcommand{\Q}{\mathbb Q}

\newcommand{\slfrak}{{\mathfrak{sl}}}

\newcommand{\shalfnote}{{\tiny \halfnote}}

\newcommand{\uc}{\mathbf{c}}

\newcommand{\uqg}{{U_q(\mathfrak{g})}}

\newcommand{\uqgp}{{U_q(\mathfrak{g'})}}

\newcommand{\vep}{\varepsilon}
\newcommand{\wght}{\mathrm{wt}}

\newcommand{\Z}{{\mathbb Z}}

\title{The bar involution for quantum symmetric pairs}
\thanks{Research supported by EPSRC grant EP/K025384/1}

\author{Martina Balagovi\'c}
\author{Stefan Kolb}
\address{Martina Balagovi\'c and Stefan Kolb, School of Mathematics and Statistics, Newcastle University, Newcastle upon Tyne NE1 7RU, UK}
\email{martina.balagovic@newcastle.ac.uk}
\email{stefan.kolb@newcastle.ac.uk}


\begin{document}

\begin{abstract}
 We construct a bar involution for quantum symmetric pair coideal subalgebras $B_{\bc,\bs}$ corresponding to involutive automorphisms of the second kind of symmetrizable Kac-Moody algebras. To this end we give unified presentations of these algebras in terms of generators and relations extending previous results by G.~Letzter and the second named author. We specify precisely the set of parameters $\bc$ for which such an intrinsic bar involution exists.
\end{abstract}

\maketitle
\section{Introduction}
Let $\gfrak$ be a symmetrizable Kac-Moody algebra and let $\uqg$ denote the corresponding quantized enveloping algebra defined over $\field(q)$, where $\field$ denotes a field of characteristic $0$. The bar involution for $\uqg$ is a fundamental ingredient in deep applications of quantum groups. It is essential in Lusztig's formulation and construction of canonical bases for $\uqg$ \cite[Chapter 14]{b-Lusztig94}. Moreover, the bar involution determines the quasi-$R$-matrix of $\uqg$, see \cite[Chapter 4]{b-Lusztig94}, and hence underlies Schur-Jimbo duality and applications of quantum groups in low dimensional topology. 

Let $\theta:\gfrak\rightarrow \gfrak$ be an involutive automorphism of the second kind \cite[III.1]{a-Levstein88} and let $\kfrak\subset \gfrak$ be the $(+1)$-eigenspace of $\theta$. The theory of quantum symmetric pairs provides quantum group analogs of the universal enveloping algebra $U(\kfrak)$ as one-sided coideal subalgebras of $\uqg$. For $\gfrak$ of finite type a comprehensive theory of quantum symmetric pair coideal subalgebras was developed by G.~Letzter in \cite{a-Letzter99a}, \cite{MSRI-Letzter}. This theory was extended to the Kac-Moody case in \cite{a-Kolb12p}. 
The original motivation behind the construction of quantum symmetric pairs was to provide new interpretations of Macdonald polynomials as zonal spherical functions on quantum symmetric spaces \cite{a-Letzter04}. In this setting, the construction of quantum symmetric pairs only relies on the structure theory of quantized enveloping algebras and the classification of involutive automorphisms of the second kind.  

Over the past year several papers appeared which consider special examples of quantum symmetric pairs in a wider representation theoretic context. M.~Ehrig and C.~Stroppel study parabolic category $\cO$ in type $D$ \cite{a-EhrigStroppel13p}. They show that translation functors yield categorifications of actions of the quantum symmetric pair coideal subalgebras corresponding to the symmetric pairs $(\glfrak(2m), \glfrak(m)\times \glfrak(m))$ and $(\glfrak(2m+1), \glfrak(m)\times \glfrak(m+1))$. In this setting the parameter $q$ appears as shift in the grading, and a bar involution naturally appears as graded duality.

In \cite{a-BaoWang13p} H. Bao and W. Wang solve the irreducible character problem in the BGG category $\cO$ for the ortho-symplectic Lie superalgebras $\mathfrak{osp}(2m{+}1|2n)$. To this end they develop a theory of canonical bases for the quantum symmetric pair coideal subalgebras corresponding to the symmetric pairs $(\glfrak(2m), \glfrak(m)\times \glfrak(m))$ and $(\glfrak(2m+1), \glfrak(m)\times \glfrak(m+1))$, respectively. In the first part of their paper Bao and Wang build this theory in astonishing similarity to Lusztig's exposition in \cite{b-Lusztig94}. A crucial ingredient in their construction is an intertwiner $\Upsilon$ between the bar involution for $\uqg$ and a new bar involution on the quantum symmetric pair coideal subalgebras under consideration. In the sequel \cite{a-BaoKuLiWang14p} it is shown that these coideal subalgebras have a geometric interpretation in terms of partial flag varieties of type $B/C$, analogous to the geometric construction of $U_q(\glfrak(N))$ by Beilinson, Lusztig, and MacPherson \cite{a-BLM90}.

The papers \cite{a-BaoWang13p} and \cite{a-EhrigStroppel13p} suggest that the bar involution is an essential tool to develop the theory of quantum symmetric pairs further. Let $B_{\bc,\bs}\subset \uqg$ be a quantum symmetric pair coideal subalgebra. The subalgebra $B_{\bc,\bs}$ is not preserved under the bar involution on $\uqg$. In \cite[0.5]{a-BaoWang13p} Bao and Wang explicitly state the existence of a bar involution for quantum symmetric pair coideal subalgebras, however, without proof and without reference to the parameters $\bc$. The aim of the present paper is to define this new bar involution on $B_{\bc,\bs}$, analogous to the bar involution on $\uqg$, and to specify the parameters for which this is possible. 

More explicitly, recall that involutive automorphisms of the second kind are determined by admissible pairs $(X,\tau)$, where $X$ denotes a subset of the set $I$ of nodes of the Dynkin diagram for $\gfrak$, and $\tau$ is a diagram automorphism, see \cite[Theorem 2.7]{a-Kolb12p}. The quantum symmetric pair coideal subalgebra $B_{\bc,\bs}=B_{\bc,\bs}(X,\tau)$ corresponding to the admissible pair $(X,\tau)$ is generated by the subalgebra $\cM_X=U_q(\gfrak_X)\subset \uqg$, a torus $U^0_\Theta$, and certain elements $B_i$ for $i\in I$ which depend on additional parameters $\bc=(c_i)_{i\in I\setminus X}, \bs=(s_i)_{i\in I\setminus X}\in \field(q)^{I\setminus X}$, see Section \ref{sec:QSP} for details. The main result of the present paper, given in Theorem \ref{thm:bar-involution}, Corollary \ref{cor:nu=1}, and Remark \ref{rem:nu-1}, states that for a suitable choice of parameters $\bc$, there exists a $\field$-algebra automorphism
\begin{align*}
 \overline{\phantom{B}}:B_{\bc,\bs}\rightarrow 
          B_{\bc,\bs}, \qquad x\mapsto \overline{x}
\end{align*}
which coincides with the usual bar involution for $\uqg$ on $\cM_X U^0_\Theta$, and which satisfies
\begin{align}\label{eq:Bibar=Bi}
  \overline{B_i}=B_i \qquad \mbox{for all $i\in I$.}
\end{align}
We explicitly describe all parameters $\bc$ for which the bar involution is well-defined. In Letzter's theory two quantum symmetric pair coideal subalgebras are considered equivalent if they are transformed into each other by a Hopf algebra automorphism of $\uqg$. We find particularly simple choices for the parameters $\bc$ in each equivalence class. 

Condition \eqref{eq:Bibar=Bi} seems natural if one attempts to generalize the constructions in \cite[Part 1]{a-BaoWang13p} to all quantum symmetric pairs. One can show that the bar involution described above allows the construction of an intertwiner $\Upsilon$ as in \cite[Section 2]{a-BaoWang13p} which in turn gives rise to a family of $B_{\bc,\bs}$-module isomorphisms between integrable highest weight modules for $\uqg$. This family provides a universal K-matrix which is a fundamental structure that has its origins in the theory of quantum integrable systems with boundary \cite{a-Cher84}, \cite{a-Sklyanin88}. The details are contained in \cite{a-BalaKolb15p}.

The algebras $B_{\bc,\bs}$ allow a presentation in terms of generators and relations \cite[Section 7]{a-Letzter03}, \cite[Section 7]{a-Kolb12p}. In particular, the generators $B_i$ satisfy inhomogeneous quantum Serre relations
\begin{align}\label{eq:qSerreBi}
  \sum_{k=0}^{1-a_{ij}}(-1)^k
  \left[\begin{matrix}1-a_{ij}\\k\end{matrix}\right]_{q_i}
  B_i^{1-a_{ij}-k} B_j B_i^k= C_{ij}(\bc) \qquad \mbox{for all $i,j\in I$}
\end{align}
for which the right hand side $C_{ij}(\bc)$ depends on the parameters $\bc$ and is of lower order in the generators $B_i$ for $i\in I$. To prove the main result of the present paper, it suffices to check whether the new bar involution as defined above on generators respects the relations \eqref{eq:qSerreBi}. Let $E_j$ for $j\in I$ denote the generators of the positive part of $\uqg$, let $K_j, K_j^{-1}$ denote the corresponding group like elements, and let $T_{w_X}$ denote Lusztig's braid group action corresponding to the longest word in the parabolic subgroup $W_X$ of the Weyl group $W$. In Theorems \ref{thm:Citaui} and \ref{thm:relsBc3} we give new unified expressions for the right hand side $C_{ij}(\bc)$ of \eqref{eq:qSerreBi} in terms of the generators $B_i$, the coefficients $\bc$, and certain elements
\begin{align}\label{eq:Zi-intro}
  \cZ_i = -s(\tau(i)) r_{\tau(i)} \big(T_{w_X}(E_{\tau(i)})\big) K_i^{-1} K_{\tau(i)} \in \cM_X U_\Theta^0\qquad \mbox{for $i\in I\setminus X$}
\end{align}
where $s(i)\in \field$ are fourth roots of unity defined by \eqref{eq:s(i)} and $r_{\tau(i)}$ denotes the skew derivation given in \cite[1.2.13]{b-Lusztig94}. Using this new presentation of $C_{ij}(\bc)$ we show in Theorem \ref{thm:bar-involution} that the new bar involution on $B_{\bc,\bs}$ is well-defined if and only if
\begin{align}\label{eq:bar-involution-intro}
  \overline{c_i \cZ_i}= q^{(\alpha_i,\alpha_{\tau(i)})} c_{\tau(i)} \cZ_{\tau(i)}
\end{align}
for essentially all $i\in I\setminus X$. One hence has to calculate the usual bar involution of $\cZ_i\in \cM_X U^0_\Theta$. We show in Proposition \ref{prop:ocZ} that
\begin{align}\label{eq:Zibar-intro}
  \overline{\cZ_i} = \nu_i \ell_i \cZ_{\tau(i)}\qquad \mbox{for $i\in I\setminus X$}
\end{align}
where $\ell_i$ is an explicitly known $q$-power and $\nu_i\in \{\pm 1\}$. If $\gfrak$ is of finite type then $\nu_i=1$, and we expect this to hold in general. 
With \eqref{eq:bar-involution-intro} and \eqref{eq:Zibar-intro} at hand it is straightforward to find the parameters $\bc$ for which the bar involution on $B_{\bc,\bs}$ exists, see Corollary \ref{cor:nu=1} and Remark \ref{rem:nu-1}.

The paper is organized as follows. In Section \ref{sec:skew-ders} we fix notations an conventions for quantum groups and prove preliminary results which will allow us to verify Relation \eqref{eq:bar-involution-intro} in Proposition \ref{prop:ocZ}. In Subsection \ref{sec:skew-derivations} we separately treat the case that $\gfrak$ is of finite type, because in this case direct calculations show that the coefficients $\nu_i$ in \eqref{eq:Zibar-intro} are always equal to $1$. An independent argument is given in Subsection \ref{sec:loc-fin-ri} in the general Kac-Moody case. In Section \ref{sec:bar-involution-QSP} we first recall the construction of quantum symmetric pairs and their presentations in terms of generators and relations. Unified formulas for the terms $C_{ij}(\bc)$ are given in Subsection \ref{sec:relations}. The somewhat technical proof of Theorem \ref{thm:Citaui} is postponed to Subsection \ref{sec:proof-relations}. With these preparations at hand we establish the existence of the bar involution in Subsection \ref{sec:existence}. This leads to the explicit description of the set of parameters $\bc$ for which a bar involution exists.  In Subsection \ref{sec:equivalence} we discuss equivalence of quantum symmetric pair coideal subalgebras with a bar involution.

\medskip

\noindent{\bf Acknowledgments.} The authors are grateful to Catharina Stroppel, Weiqiang Wang, and two anonymous referees for useful comments.

\medskip

\section{Skew derivations and braid group action}\label{sec:skew-ders}
To show the existence of the bar involution for quantum symmetric pair coideal subalgebras we need to know the behavior of the elements $\cZ_i$ given by \eqref{eq:Zi-intro} under the usual bar involution for $\uqg$. The purpose of the present section is to provide the tools necessary to give a formula for $\overline{\cZ_i}$ in Proposition \ref{prop:ocZ}. Along the way we will fix notations and conventions for quantum groups.
\subsection{Preliminaries}
All through this paper the symbol $\N$ denotes the positive integers, $\N_0=\N\cup\{0\}$, we write $\Z$ to denote the integers, $\Q$ are the rational numbers, and $\field$ denotes an algebraically closed field of characteristic $0$. We write $\field(q)$ to denote the field of rational functions in an indeterminate $q$.  

Let $\gfrak=\gfrak(A)$ be a symmetrizable Kac-Moody algebra defined over $\field$ corresponding to a generalized Cartan matrix $A=(a_{ij})_{i,j\in I}$ where $I$ denotes some finite set. By definition, there exists a diagonal matrix $D=\mathrm{diag}(\epsilon_i\,|\,i\in I)$ with coprime entries $\epsilon_i\in \N$ such that $DA$ is a symmetric matrix. Let $\hfrak$ be the Cartan subalgebra of $\gfrak$ and let $\Pi=\{\alpha_i\,|\,i\in I\}\subset \hfrak^\ast$ and $\Pi^\vee=\{h_i\,|\,i\in I\}\subset \hfrak$ denote the sets of simple roots and simple coroots, respectively. One has in particular $\alpha_j(h_i)=a_{ij}$. Let $(\cdot,\cdot)$ denote a nondegenerate, symmetric, bilinear form on $\hfrak^\ast$ such that
\begin{align*}
  (\alpha_i,\alpha_j)=\epsilon_i a_{ij} \qquad \mbox{for all $i,j\in I$.}
\end{align*}
The quantized enveloping algebra $\uqg$ of $\gfrak$ is the $\field(q)$ algebra generated by elements $E_i$, $F_i$, $K_i$, $K_i^{-1}$ for all $i\in I$ and relations
\begin{enumerate}
  \item $K_i K_j=K_j K_i$ and $K_i K_i^{-1}=1=K_i^{-1} K_i$ for all $i,j\in I$. 
  \item $K_i E_j= q^{(\alpha_i,\alpha_j)} E_j K_i$ for all $i,j\in I$.
  \item $K_i F_j= q^{-(\alpha_i,\alpha_j)} F_j K_i$ for all $i,j\in I$.
  \item $\displaystyle E_i F_j - F_j E_i=\delta_{ij} \frac{K_i-K_i^{-1}}{q_i-q_i^{-1}}$ for all $i,j\in I$ where $q_i=q^{\epsilon_i}$.
  \item \label{q-Serre} The quantum Serre relations given in \cite[3.1.1.(e)]{b-Lusztig94}.
\end{enumerate}
For any $i,j\in I$ let $F_{ij}(x,y)$ denote the noncommutative polynomial in two variables defined by
\begin{align}\label{eq:Fij-def}
  F_{ij}(x,y)=\sum_{n=0}^{1-a_{ij}}(-1)^n
  \left[\begin{matrix}1-a_{ij}\\n\end{matrix}\right]_{q_i}
  x^{1-a_{ij}-n}yx^n.  
\end{align}
where $\left[ \abv{1-a_{ij}}{n} \right]_{q_i}$ denotes the $q_i$-binomial coefficient defined for instance in \cite[1.3.3]{b-Lusztig94}. By \cite[Corollary 33.1.5]{b-Lusztig94}, the quantum Serre relations in (\ref{q-Serre}) can be written in the form
\begin{align}\label{eq:q-Serre}
  F_{ij}(E_i,E_j)=F_{ij}(F_i,F_j)=0\qquad \mbox{for all $i,j\in I$}.
\end{align}
The algebra $\uqg$ is a Hopf algebra with coproduct $\kow$, counit $\vep$, and antipode $S$ given by
\begin{align}
  \kow(E_i)&=E_i \ot 1 + K_i\ot E_i,& \vep(E_i)&=0, & S(E_i)&=-K_i^{-1} E_i,\label{eq:E-copr}\\
  \kow(F_i)&=F_i\ot K_i^{-1} + 1 \ot F_i,&\vep(F_i)&=0, & S(F_i)&=-F_i K_i,\label{eq:F-copr}\\
  \kow(K_i)&=K_i \ot K_i,& \vep(K_i)&=1, & S(K_i)&=K_i^{-1}\label{eq:K-copr}
\end{align}
for all $i\in I$. 
\begin{rema}
  In \cite{a-Kolb12p} the Hopf algebra $\uqg$ was denoted by the symbol $\uqgp$ because it is a $q$-deformation of the enveloping algebra of the derived Lie algebra $\gfrak'=[\gfrak,\gfrak]$ of $\gfrak$. If $\gfrak$ is of finite type then $\gfrak=\gfrak'$. Here we follow the common convention to work only with the $q$-analog of the derived algebra and drop the distinction in notation.  In \cite[Definition 5.6]{a-Kolb12p} quantum symmetric pairs are defined as coideal subalgebras of $\uqg$ as defined here. An extension to the larger Hopf algebra may be obtained by adding group-like elements to the coideal subalgebra, see \cite[Remark 5.8]{a-Kolb12p}.
\end{rema}
As usual we write $U^+$, $U^-$, and $U^0$ to denote the $\field(q)$-subalgebra of $\uqg$ generated by $\{E_i\,|\,i\in I\}$, $\{F_i\,|\,i\in I\}$, and $\{K_i, K_i^{-1}\,|\,i\in I\}$, respectively.
Let $Q= \Z \Pi$ denote the root lattice of $\gfrak$ and define $Q^+= \N_0 \Pi$ and $Q^-=-\N_0 \Pi$. For $\beta=\sum_{i\in I} n_i \alpha_i\in Q$ we write $K_\beta=\prod_{i\in I}K_i^{n_i}$. The algebra $U^+$ has a $Q^+$-grading given by $\deg(E_i)=\alpha_i$. Similarly, $U^-$ has a $Q^-$-grading given by $\deg(F_i)=-\alpha_i$. Via the triangular decomposition
\begin{align*}
  \uqg \cong U^+\ot U^0 \ot U^-
\end{align*}
one obtains a $Q$-algebra grading on $\uqg$ such that
\begin{align*}
  \uqg_\gamma = \mathop{\bigoplus}_{\alpha+\beta=\gamma} U^+_\alpha U^0 U^-_\beta \qquad \mbox{for any $\gamma \in Q$.}
\end{align*}
The Weyl group $W$ of $\gfrak$ is the Coxeter group generated by elements $\{s_i\,|\,i\in I\}$ and defining relations 
\begin{align}\label{Coxeter}
  s_i^2=1, \qquad (s_i s_j)^{m_{ij}}=1 \qquad \mbox{for all $i,j\in I$}
\end{align}
where $m_{ij}=2,3,4,6,$ and $\infty$ if
$a_{ij}a_{ji}=0,1,2,3$, and $\ge 4$, respectively. The group $W$ acts on $\hfrak$ and $\hfrak^\ast$ by
\begin{align*}
  s_i(h)&=h-\alpha_i(h)h_i, \quad \mbox{ for all $i\in I$, $h\in \hfrak$},\\
  s_i(\alpha)&=\alpha-\alpha(h_i)\alpha_i \quad \mbox{ for all $i\in I$, $\alpha\in \hfrak^\ast$},
\end{align*}
respectively. For any $i\in I$ let $T_i$ denote the algebra automorphism of $\uqg$ denoted by $T_{i,1}''$ in \cite[37.1]{b-Lusztig94}. The automorphisms $T_i$ satisfy braid relations. This implies that for any $w\in W$ one has an algebra automorphism
\begin{align*}
  T_w:\uqg \rightarrow \uqg, \qquad T_w=T_{i_1} T_{i_2} \dots T_{i_k}
\end{align*}
where $w=s_{i_1} s_{i_2} \dots s_{i_k}$ is any reduced expression of $w$.

\subsection{The skew derivations $r_i$ and ${}_i r$}\label{sec:skew-derivations}
In Proposition \ref{prop:finite-st} and \ref{prop:sigma-tau-ri} we formulate an invariance property which will be crucial to prove the existence of the bar involution for quantum symmetric pairs. This invariance property combines Lusztig's braid group action for admissible pairs with skew derivations $r_i,\, {}_ir:U^+\rightarrow U^+$ defined in \cite[1.2.13]{b-Lusztig94}. The skew derivations $r_i$ and ${}_ir$ describe the first order behavior of elements of $U^+$ under the coproduct $\kow$. More explicitly, for every $x\in U^+_\beta$, $\beta\in Q^+$, there exist uniquely determined elements $r_i(x), {}_ir(x)\in U^+$ such that
\begin{align}
  \kow(x) &= x\ot 1 +\sum_{i\in I} r_i(x)K_i \ot E_i + \mathrm{(rest)}_1\label{eq:ri-def}\\
  \kow(x) &= K_\beta\ot x +\sum_{i\in I} E_i K_{\beta-\alpha_i}\ot {_i}r(x) + \mathrm{(rest)}_2\label{eq:ir-def}
\end{align}
where 
$  (\mathrm{rest})_1\in \sum_{0<\alpha\in Q^+\setminus\Pi} U^+_{\beta-\alpha}K_\alpha\otimes U^+_\alpha$ and $(\mathrm{rest})_2\in \sum_{0<\alpha\in Q^+\setminus\Pi} U^+_{\alpha}K_{\beta-\alpha}\otimes U^+_{\beta-\alpha}.$
By \cite[Lemma 6.17]{b-Jantzen96}, \cite[Proposition 3.1.6]{b-Lusztig94} one has
\begin{align}\label{eq:ri-Fcomm}
  x F_i - F_i x = \frac{r_i(x)K_i - K_i^{-1}{}_ir(x)}{q_i-q_i^{-1}} \qquad \mbox{for all $i\in I$}. 
\end{align}
This can be reformulated as follows. For  $x=E_{i_1}\dots E_{i_r}\in U^+$ one has
\begin{align}\label{eq:rix-explicit}
  r_i(x)=\sum_{j=1\atop i_j=i}^r E_{i_1} \dots E_{i_{j-1}}K_i E_{i_{j+1}} \dots E_{i_r} K_i^{-1}.
\end{align}
As in \cite[3.1.3]{b-Lusztig94} let $\sigma:\uqg\rightarrow \uqg$ denote the $\field(q)$-algebra antiautomorphism determined by 
\begin{align*}
 \sigma(E_i)&=E_i, & \sigma(F_i)&=F_i,& \sigma(K_i)&=K_i^{-1} \qquad \mbox{for all $i\in I$}.
\end{align*}
Here the requirement that $\sigma$ is an algebra antiautomorphism means that $\sigma(xy)=\sigma(y)\sigma(x)$ for all $x,y\in \uqg$.
Relation \eqref{eq:ri-Fcomm} implies that the map $\sigma$ intertwines the skew derivations $r_i$ and ${}_i r$ as follows
\begin{align}\label{eq:sigma-ri}
  \sigma \circ r_i = {}_i r \circ \sigma\qquad \mbox{for all $i\in I$}.
\end{align}
Moreover, by \cite[37.2.4]{b-Lusztig94} the map $\sigma$ intertwines the automorphism $T_i$ with its inverse
\begin{align}\label{eq:Ti-sigma}
  T_i \circ \sigma = \sigma \circ T_i^{-1} \qquad \mbox{for all $i\in I$}.
\end{align}

Let $\Aut(A)$ denote the group of all permutations $\tau$ of the set $I$ for which $a_{ij}=a_{\tau(i)\tau(j)}$ for all $i,j\in I$. For any subset $X\subseteq I$ define
\begin{align*}
  \Aut(A,X)=\{\tau\in \Aut(A)\,|\,\tau(X)=X\}.
\end{align*}
We recall the notion of an admissible pair from \cite[Definition 2.3]{a-Kolb12p}. For any subset $X\subseteq I$ of finite type let $w_X\in W$ denote the longest element in the parabolic subgroup $W_X \subseteq W$ corresponding to $X$. Moreover, let $\gfrak_X\subseteq \gfrak$ denote the semisimple Lie algebra corresponding to $X$ and let $\Phi_X$ denote its root system. We write
$\rho_X$ and $\rho^\vee_X$ to denote the half sum of positive roots and of positive coroots for the root system $\Phi_X$, respectively.
\begin{defi}\label{def:admissible}
  A pair $(X,\tau)$ consisting of a subset $X\subseteq I$ of finite type and an
  element $\tau\in \Aut(A,X)$ is called admissible if the following
  conditions are satisfied:
  \begin{enumerate}
    \item $\tau^2=\id_I$.
    \item \label{adm2} The action of $\tau$ on $X$ coincides with the action of $-w_X$.
    \item \label{adm3} If $j\in I\setminus X$ and $\tau(j)=j$ then
         $\alpha_j(\rho_X^\vee)\in \Z$. 
   \end{enumerate}
\end{defi}
As proved in \cite{a-KW92}, admissible pairs parametrize involutive automorphism of the second kind of $\gfrak$ up to conjugation, see also \cite[Theorem 2.7]{a-Kolb12p}.
If $\gfrak$ is of finite type, then the pair $(X=I,\tau=-w_X)$ is always admissible. The remaining admissible pairs for indecomposable $\gfrak$ of finite type are given by the table in \cite[p.32]{a-Araki62} where the black dots correspond to elements in the set $X$ and the arrows indicate the diagram automorphism $\tau$. This classification allows us to verify the following proposition case by case. 
\begin{prop}\label{prop:finite-st}
  Assume that $\gfrak$ is of finite type and let $(X,\tau)$ be an admissible pair for $\gfrak$. Then
  \begin{align}\label{eq:sigma-tau}
    \sigma\circ \tau(r_i(T_{w_X}(E_i)))=r_{i}(T_{w_X}(E_{i}))
  \end{align} 
  holds for all $i\in I\setminus X$.
\end{prop}
\begin{proof}
  If $X=\emptyset$ then both sides of Equation \eqref{eq:sigma-tau} are equal to $1$. If $X\neq \emptyset$ then the essential step in the proof is to find a suitable reduced expression for $w_X$. For fixed $i\in I\setminus X$ this reduced expression is of the form $w_X=w_X' w_X''$ where $w_X''\in W_X$ is of maximal length with $w_X''(\alpha_i)=\alpha_i$ and the lengths of the three elements are related by $l(w_X)=l(w_X')+l(w_X'')$. We prove Formula \eqref{eq:sigma-tau} for $\gfrak$ and $(X,\tau)$ of types $AIV$, $BII$, and $DII$ in Araki's table \cite[p.~32]{a-Araki62}. All other cases are treated similarly and are left to the reader.
  
\noindent  {\bf Type $AIV$:} Here $\gfrak=\slfrak_{n+1}(\C)$ and $I=\{1,2,\dots,n\}$. The admissible pair is given by $X=\{2,3,\dots,n-1\}$ and $\tau$ is the nontrivial diagram automorphism $\tau(i)=n+1-i$. There are only two nodes in $I\setminus X$, namely the nodes labeled by $1$ and $n$. Using the notation $[x,y]_p=xy-pyx$ and the reduced expression
\begin{align*}
  w_X=(s_{n-1}s_{n-2 }\dots s_2)(s_{n-1} s_{n-2} \dots s_3) \dots (s_{n-1} s_{n-2}) s_{n-1}
\end{align*}
one calculates 
\begin{align}
  r_1(T_{w_X}(E_1))&\stackrel{\phantom{\eqref{eq:rix-explicit}}}{=}r_1(T_{n-1} T_{n-2} \dots T_2 (E_1))\nonumber\\
          &\stackrel{\phantom{\eqref{eq:rix-explicit}}}{=}r_1(T_{n-1} T_{n-2} \dots T_3 ([E_2,E_1]_{q^{-1}}))\nonumber\\
          &\stackrel{\phantom{\eqref{eq:rix-explicit}}}{=}r_1([T_{n-1} \dots T_3(E_2),E_1]_{q^{-1}}))\nonumber\\  
          &\stackrel{\eqref{eq:rix-explicit}}{=}(1-q^{-2})T_{n-1} \dots T_3(E_2) \label{eq:r1twXE1}
\end{align}
For symmetry reasons one also has
\begin{align*}
  r_n(T_{w_X}(E_n))=\tau(r_1(T_{w_X}(E_1)))=(1-q^{-2}) T_2 T_3 \dots T_{n-2}(E_{n-1})
\end{align*}
Now define $w=s_{n-2}s_{n-3}\dots s_2 s_{n-1} s_{n-2}\dots s_3$ and observe that the right hand side  is a reduced expression and $w(\alpha_2)=\alpha_{n-1}$. Hence, using \cite[Proposition 8.20]{b-Jantzen96} and \eqref{eq:r1twXE1}, one obtains 
\begin{align}\label{eq:inbetween}
  r_1(T_{w_X}(E_1))&= (1-q^{-2}) T_2^{-1} T_3^{-1}\dots T_{n-2}^{-1}(E_{n-1}).      
\end{align}
Using Equations \eqref{eq:inbetween} and \eqref{eq:Ti-sigma} and the relation $\tau(w_X)=w_X$ one gets
\begin{align*}
  \sigma \circ \tau (r_n(T_{w_X}(E_n)))&= \sigma(r_1(T_{w_X}(E_1)))\\
     &=(1-q^{-2})\sigma(T_2^{-1} T_3^{-1}\dots T_{n-2}^{-1}(E_{n-1}))\\
     &=(1-q^{-2}) T_2 T_3 \dots T_{n-2}(E_{n-1})\\
     &=  r_n(T_{w_X}(E_n)). 
\end{align*}
The relation $\sigma \circ  \tau (r_1(T_{w_X}(E_1)))=  r_1(T_{w_X}(E_1))$ holds again for symmetry reasons.

\noindent  {\bf Type $BII$:} In this case $\gfrak$ is of type $B_n$ with $I=\{1,2,\dots,n\}$ where $\alpha_n$ denotes the short root. The admissible pair is given by $X=\{2,3,\dots,n\}$ and $\tau=\id$. To shorten notation, for any $1\le k <  n$, write
\begin{align}
  s_{k--n--k}=s_k s_{k+1} \dots s_{n-1} s_n s_{n-1} \dots s_{k+1} s_k,\label{eq:sknk}\\
  T_{k--n--k}=T_k T_{k+1} \dots T_{n-1} T_n T_{n-1} \dots T_{k+1} T_k.\nonumber
\end{align}
The longest element of $W_X$ is given by
\begin{align*}
  w_X= s_{2--n--2}\, s_{3--n--3} \dots s_{(n-1)--n--(n-1)}\, s_n.
\end{align*}
Moreover, the expression
\begin{align*}
  s_{3--n--3}\, s_{2--n--2} = s_{3--n--3}\,s_2\, s_{3--n--3} s_2
\end{align*}
is reduced. As $s_{3--n--3}\,s_2\, s_{3--n--3}(\alpha_2)=\alpha_2$ one obtains, using again \cite[Proposition 8.20]{b-Jantzen96}, that
\begin{align}\label{eq:h1}
   T_2 T_{3--n--3}(E_2)=T_{3--n--3}^{-1}(E_2)
\end{align}
Now one calculates
\begin{align}
  r_1(T_{w_X}(E_1))&\stackrel{\phantom{\eqref{eq:h1}}}{=} r_1(T_{2--n--2}(E_1)) \label{eq:calc-BII-1} \\
      &\stackrel{\phantom{\eqref{eq:h1}}}{=} r_1(T_2 T_{3--n--3}[E_2,E_1]_{q^{-2}})\nonumber\\
      &\stackrel{\phantom{\eqref{eq:h1}}}{=}r_1[T_2 T_{3--n--3}(E_2), T_2(E_1)]_{q^{-2}}\nonumber\\
      &\stackrel{\eqref{eq:h1}}{=}r_1[T_{3--n--3}^{-1}(E_2), [E_2,E_1]_{q^{-2}}]_{q^{-2}}\nonumber\\
      &\stackrel{\eqref{eq:rix-explicit}}{=}(1-q^{-4})[T_{3--n--3}^{-1}(E_2), E_2]_{q^{-4}}.\nonumber
\end{align}
Hence Equation \eqref{eq:Ti-sigma} implies 
\begin{align}
  \sigma(r_1(T_{w_X}(E_1)))&=T_{3--n--3}(r_1(T_{w_X}(E_1)))\label{eq:calc-BII-2}\\      
     &=r_1(T_{3--n--3}T_{w_X}(E_1))=r_1(T_{w_X}(E_1))\nonumber
\end{align}
which concludes the proof in the case $BII$.

\noindent  {\bf Type $DII$:} The proof of the proposition in the case $DII$ is very similar to the proof given in the case $BII$ above. Indeed, let $\gfrak$ be of type $D_n$ with $X=\{2,3,\dots,n\}$ where the simple roots a labeled as in \cite[4.8]{b-Kac1}. For any $k=1,\dots,n-2$ define
\begin{align*}
  s'_{k--n--k}&=s_k s_{k+1} \dots s_{n-2} s_{n-1} s_n s_{n-2}\dots s_{k+1} s_k,\\
  T'_{k--n--k}&=T_k T_{k+1} \dots T_{n-2} T_{n-1} T_n T_{n-2} \dots T_{k+1} T_k.\nonumber
\end{align*}
The longest element in $W_X$ is given by
\begin{align*}
  w_X=s'_{2--n--2}s'_{3--n--3}\dots s'_{(n-2)--n--(n-2)} s_{n-1} s_n
\end{align*}
Moreover, the expression 
\begin{align*}
  s'_{3--n--3}s'_{2--n--2} = s'_{3--n--3} s_2 s'_{3--n--3} s_2
\end{align*}
is reduced and $s'_{3--n--3} s_2 s'_{3--n--3}(\alpha_2)=\alpha_2$ which implies that
\begin{align*}
  T_2 T'_{3--n--3}(E_2)= (T'_{3--n--3})^{-1}(E_2).
\end{align*}
Now one calculates as in \eqref{eq:calc-BII-1} and \eqref{eq:calc-BII-2}.
\end{proof}
\subsection{Local finiteness and $r_i(T_{w_X}(E_i))$}\label{sec:loc-fin-ri}
For any subset $X\subset I$ of finite type let $\cM_X=U_q(\gfrak_X)$ denote the subalgebra of $\uqg$ generated by the elements $E_i$, $F_i$, $K_i^{\pm 1}$ for all $i\in X$. Correspondingly, let $\cM_X^+$ and $\cM_X^-$ denote the subalgebras of $\cM_X$ generated by the elements in the sets $\{E_i\,|\,i \in X\}$ and  $\{F_i\,|\,i \in X\}$, respectively.

Avoiding casework, we will show in this subsection that Proposition \ref{prop:finite-st} holds up to a sign even if we drop the assumption that $\gfrak$ is of finite type.  To achieve this we will make use of A.~Joseph and G.~Letzter's theory of the locally finite part applied to $\cM_X$. Using Sweedler notation in the form $\kow(u)=u_{(1)}\ot u_{(2)}$ for $u\in \uqg$, the left adjoint action of $u$ on $x\in \uqg$ is given by 
\begin{align*}
  \ad(u)(x)=u_{(1)}x S(u_{(2)}).
\end{align*}
One has in particular
\begin{align}\label{eq:left-adjointE}
  \ad(E_i)(x)=E_i x - K_i x K_i^{-1} E_i = K_i(K_i^{-1} E_i x - x K_i^{-1} E_i).  
\end{align}
For any $x\in \uqg$ the space $\ad(\cM_X)(x)$ is an $\cM_X$-module under the left adjoint action. 
Let $P_X^+$ denote the set of dominant integral weights for the semisimple Lie algebra $\gfrak_X$.  For each $\lambda\in P_X^+$ let $V_X(\lambda)$ denote the irreducible highest weight left $\cM_X$-module of weight $\lambda$ and let $V_X(\lambda)^\ast$ denote its dual left $\cM_X$-module. Up to translation into our present setting, the following result is contained in \cite{a-JoLet2}, \cite{a-Caldero93}. See \cite[Section 1.3 and Proposition 4.12]{a-KolbStok09} for details and for further references.
  \begin{prop}\label{prop:loc-fin}
    Let $\lambda\in P_X^+$ and assume that $K\in U^0$ satisfies $K E_j K^{-1}= q^{-(\lambda,\alpha_j)} E_j$ for all $j\in X$. 
  \begin{enumerate}
    \item The left $\cM_X$-module $\ad(\cM_X)(K^2)$ is isomorphic to $V_X(\lambda)\otimes V_X(\lambda)^\ast$. In particular, it is finite dimensional.
    \item The space
      \begin{align*}
         \{x\in U^+_{\lambda-w_X(\lambda)}K^2\,|\,\ad(E_j)(x)=0 \mbox{ for all $j\in X$}\}
      \end{align*}     
      is one-dimensional and is contained in $\ad(\cM_X)(K^2)$.
   \end{enumerate}
\end{prop}
To apply Proposition \ref{prop:loc-fin} to our setting recall from \cite[(4.4)]{a-Kolb12p} that for every $i\in I\setminus X$ one has $T_{w_X}(E_i)=a_i \ad(Z_i^+)(E_i)$ for some $a_i\in \field(q)$ and some monomial $Z_i^+\in U^+_{w_X(\alpha_i)-\alpha_i}$. By \eqref{eq:ri-Fcomm} this implies that
\begin{align}\label{eq:ri-ad}
  r_i(T_{w_X}(E_i)) = a_i \ad(Z_i^+)(K_i^2) K_i^{-2}.
\end{align}
It now follows from Proposition \ref{prop:loc-fin} that  $r_i(T_{w_X}(E_i))K_i^2$ spans the unique one-dimensional subspace of highest weight vectors of weight $w_X(\alpha_i)-\alpha_i$ inside $\cM_X^+ K_{i}^2$. This allows us to verify that Proposition \ref{prop:finite-st} holds up to a sign also for $\gfrak$ of infinite type.
\begin{prop}\label{prop:sigma-tau-ri}
  Let $(X,\tau)$ be an admissible pair for the symmetrizable Kac-Moody algebra $\gfrak$. For every $i\in I\setminus X$ there exists $\nu_i\in \{-1,1\}$ such that
  \begin{align*}
    \sigma\circ \tau(r_i(T_{w_X}(E_i)))=\nu_i r_{i}(T_{w_X}(E_{i})).
  \end{align*} 
  Moreover, $\nu_i=\nu_{\tau(i)}$ for all $i\in I\setminus X$.
\end{prop}
\begin{proof}
By \eqref{eq:left-adjointE} the fact that $r_i(T_{w_X}(E_i))K_i^2$ is a highest weight vector for the left adjoint action of $\cM_X$ is equivalent to
\begin{align*}
   K_j^{-1} E_j\, r_i(T_{w_X}(E_i)) K_i^{2}  - r_i(T_{w_X}(E_i)) K_i^{2} K_j^{-1} E_j =0 \qquad \mbox{for all $j\in X$.}
\end{align*} 
This can be rewritten as
\begin{align*}
   E_j\, r_i(T_{w_X}(E_i))  - q^{(\alpha_j , \alpha_i+w_X\alpha_i)} r_i(T_{w_X}(E_i))  E_j =0 \qquad \mbox{for all $j\in X$.}
\end{align*}
 Using the Weyl group invariance of the bilinear form and the fact that $w_X(\alpha_j)=-\alpha_{\tau(j)}$ for all $j\in X$, one can simplify the $q$-exponent and one obtains
\begin{align*}
  E_j\, r_i(T_{w_X}(E_i))  - q^{(\alpha_j , \alpha_i - \alpha_{\tau(i)})} r_i(T_{w_X}(E_i))  E_j =0 \qquad \mbox{for all $j\in X$.}
\end{align*}
Applying $\sigma\circ \tau$ to the above equation one obtains
\begin{align*}
    \sigma\circ \tau\big(r_i(T_{w_X}(E_i))\big)\,E_{\tau(j)}  - q^{(\alpha_j , \alpha_i - \alpha_{\tau(i)})} E_{\tau(j)}\sigma\circ \tau\big(r_i(T_{w_X}(E_i))\big)  =0
\end{align*}
and hence
\begin{align*}
   E_j \sigma\circ \tau(r_i(T_{w_X}(E_i)))  - q^{(\alpha_j ,\alpha_i-\alpha_{\tau(i)})} \sigma\circ \tau(r_i(T_{w_X}(E_i)))  E_j =0 \qquad \mbox{for all $j\in X$.}
\end{align*}
Translating this back, we obtain that
\begin{align*}
  \sigma\circ\tau(r_i(T_{w_X}E_i)) K_i^2 \in \cM_X^+ K_i^2
\end{align*}
is a highest weight vector for the adjoint action of $\cM_X$ on $\cM_XU^0$ of weight $w_X(\alpha_i)-\alpha_i$. By Proposition \ref{prop:loc-fin}.(2) and the subsequent remark this implies that
\begin{align}\label{eq:PinZ}
   \sigma\circ\tau(r_i(T_{w_X}E_i))=\nu_i r_{i}(T_{w_X}E_{i})
\end{align}
for some $\nu_i\in \field(q)$. As $(\sigma\circ \tau)^2$ is the identity map, one obtains $\nu_i^2= 1$ which implies that $\nu_i\in \{-1,1\}$. The relation $\nu_i=\nu_{\tau(i)}$ is obtained by application of $\tau$ to both sides of \eqref{eq:PinZ}. 
\end{proof}
\begin{rema}
  The proof of Proposition \ref{prop:sigma-tau-ri} only uses Conditions (1) and (2) from Definition \ref{def:admissible}. The proposition hence holds for a wider class of pairs $(X,\tau)$. Here we restrict to admissible pairs because we are interested in applications to quantum symmetric pairs. 
\end{rema}
In view of Proposition \ref{prop:finite-st} it is to be expected that the coefficients $\nu_i$ in Proposition \ref{prop:sigma-tau-ri} are always equal to $1$. We have checked this for various examples where $\gfrak$ is not of finite type and it always turned out to hold true. However, a general argument is still evasive. 
\begin{conj}\label{conj:sigma-tau-ri}
  Let $(X,\tau)$ be an admissible pair for the symmetrizable Kac-Moody algebra $\gfrak$. The signs $\nu_i$ in Proposition \ref{prop:sigma-tau-ri} satisfy $\nu_i=1$ for all $i\in I\setminus X$. 
\end{conj}
\begin{rema}
  The main results of this paper, in particular Theorem \ref{thm:bar-involution}, are formulated independently of the signs $\nu_i$ for $i\in I\setminus X$. The relevance of Conjecture \ref{conj:sigma-tau-ri} will become apparent in Corollary \ref{cor:nu=1} and the subsequent Remarks \ref{rem:q-exponent-even} and \ref{rem:nu-1}. If $\nu_i=1$ for all $i\in I \setminus X$ then a bar involution for quantum symmetric pair coideal subalgebras $B_{\bc,\bs}$ exists for parameters $\bc=(c_i)_{i\in I\setminus X}$ where all $c_i$ are powers of $q$, see Remark \ref{rem:q-exponent-even}. This is desirable in view of the specialization properties of $B_{\bc, \bs}$, see \cite[Section 10]{a-Kolb12p}. If $\nu_i=-1$ for some $i\in I\setminus X$ then one can still choose parameters $\bc=(c_i)_{i\in I\setminus X}$ such that a bar involution for $B_{\bc,\bs}$ exists, but now some of the parameters specialize to $0$, see Remark \ref{rem:nu-1}. This does not lead to the desired specialization of $B_{\bc,\bs}$ at $q=1$. The conjecture expresses our belief that this situation will never occur. 
\end{rema}
\subsection{The bar involution for $\uqg$}
With the preparation provided by the previous subsection it is possible to describe the behavior of the elements $r_i(T_{w_X}(E_i))$ under the usual bar involution for $\uqg$. Recall that the bar involution for $\uqg$ is the $\field$-algebra automorphism $\overline{\phantom{bl}}:\uqg\rightarrow \uqg$, $u\mapsto \overline{u}$, defined by
\begin{align*}
  \overline{q}&=q^{-1}, & \overline{E_i}&=E_i, & \overline{F_i}&=F_i,& \overline{K_\beta}&=K_\beta^{-1}.
\end{align*}
By \cite[Lemma 1.2.14]{b-Lusztig94} the bar involution intertwines the skew derivations $r_i$ and ${}_i r$ up to a factor
\begin{align}\label{eq:bar-ri}
  \overline{r_i(u)} = q^{(\alpha_i,\alpha_i-\beta)} {}_i r(\overline{u})\quad \mbox{for all $i\in I$,  $u\in U^+_\beta$}.
\end{align}
Moreover, the bar involution is also compatible with the four braid group actions $T'_{i,1}$, $T''_{i,1}$, $T'_{i,-1}$, $T''_{i,-1}$ defined in \cite[37.1]{b-Lusztig94} in the following sense
\begin{align}\label{eq:bar-Ti}
  \overline{T''_{i,e}(u)}=T''_{i,-e}(\overline{u}),\quad
  \overline{T'_{i,e}(u)}=T'_{i,-e}(\overline{u}),\quad \mbox{for all $i\in I$,  $u\in \uqg$, $e\in \{\pm 1\}$.}
\end{align}
Recall that in the conventions of the present paper we have $T_i=T''_{i,1}$ and $T^{-1}_i = T'_{i,-1}$. The relation between $T''_{i,e}$ and $T'_{i,e}$ is given by 
\begin{align}\label{eq:T21}
  T''_{i,e}(u) = (-1)^n q_i^{en} T'_{i,e}(u)
\end{align}
for all $u\in \uqg$ with $K_i u K_i^{-1}=q_i^n u$ \cite[37.2.4]{b-Lusztig94}. In the proof of the following Lemma we will make use of all four braid group actions. If $w=s_{i_1}\dots s_{i_k}$ is a reduced expression for $w\in W$ then we will write $T'_{w,e}=T'_{i_1,e}\dots T'_{i_k,e}$ and similarly $T''_{w,e}=T''_{i_1,e}\dots T''_{i_k,e}$. Recall that $\rho_X$  and  $\rho_X^\vee$ denote the half sums of positive roots and of positive coroots of $\Phi_X$, respectively.
\begin{lem}
  Let $(X,\tau)$ be an admissible pair for $\gfrak$. For any $i\in I\setminus X$ one has
  \begin{align*}
    \overline{r_i(T_{w_X}(E_i))} = (-1)^{2\alpha_i(\rho_X^\vee)} q^{(\alpha_i,\,\alpha_i-w_X(\alpha_i)-2\rho_X)} \sigma\circ \tau(r_{\tau(i)}(T_{w_X}(E_{\tau(i)}))).
  \end{align*}
\end{lem}
\begin{proof}
  Using the above properties of the bar involution and of Lusztig's automorphisms we calculate
  \begin{align}
    \overline{r_i(T_{w_X}(E_i))} &\stackrel{\eqref{eq:bar-ri}}{=} q^{(\alpha_i,\alpha_i-w_X(\alpha_i))} {}_i r(\overline{T_{w_X}(E_i)}) \nonumber\\
      &\stackrel{\eqref{eq:bar-Ti}}{=} q^{(\alpha_i,\,\alpha_i-w_X(\alpha_i))} {}_i r(T''_{w_X,-1}(E_i)). \label{eq:step1}
  \end{align}
  For any reduced expression $w_X=s_{i_1}\dots s_{i_k}$ define $\beta_l=s_{i_k} s_{i_{k-1}}\dots s_{i_{l+1}}\alpha_{i_l}$ for $l=1,\dots,k$. The set $\{\beta_l\,|\,l=1,\dots,k\}$ consists of all positive roots of the root system $\Phi_X$. Hence
  \begin{align*}
    \sum_{l=1}^k (\alpha_{i_l},s_{i_{l+1}}\dots s_{i_k}\alpha_i)=\sum_{l=1}^k (\beta_l,\alpha_i)=(2\rho_X,\alpha_i).
  \end{align*} 
In view of the above relation, Equation \eqref{eq:T21} implies 
  \begin{align*}
    T''_{w_X,e}(E_i) = (-1)^{2\alpha_i(\rho_X^\vee)} q^{e(2\rho_X,\alpha_i)} T'_{w_X,e}(E_i).
  \end{align*}
Inserting this for $e=-1$ into \eqref{eq:step1} we obtain
\begin{align*}
   \overline{r_i(T_{w_X}(E_i))}    &\stackrel{\phantom{\eqref{eq:bar-ri}}}{=}   
       (-1)^{2\alpha_i(\rho_X^\vee)}  q^{(\alpha_i,\,\alpha_i-w_X(\alpha_i)-2\rho_X)} 
           {}_i r(T'_{w_X,-1}(E_i)) \\
   &\stackrel{\eqref{eq:Ti-sigma}}{=}(-1)^{2\alpha_i(\rho_X^\vee)}  q^{(\alpha_i,\,\alpha_i-w_X(\alpha_i)-2\rho_X)} {}_i r(\sigma(T_{w_X}(E_i)))\\
   &\stackrel{\eqref{eq:sigma-ri}}{=}(-1)^{2\alpha_i(\rho_X^\vee)}  q^{(\alpha_i,\,\alpha_i-w_X(\alpha_i)-2\rho_X)} \sigma\circ \tau(r_{\tau(i)}(T_{w_X}(E_{\tau(i)})))
\end{align*}  
which concludes the proof of the lemma.
\end{proof}
Using Proposition \ref{prop:sigma-tau-ri} the formula in the above Lemma can be further simplified. To shorten notation we define
\begin{align}\label{eq:li-def}
  \ell_i = q^{(\alpha_i,\,\alpha_i-w_X(\alpha_i)-2\rho_X)} \qquad \mbox{for all $i\in I\setminus X$.}
\end{align} 
Observe that $\ell_i=\ell_{\tau(i)}$. Recall the definition of the signs $\nu_i$ from Proposition \ref{prop:sigma-tau-ri}.
\begin{cor}\label{cor:bar-ri}
  Let $(X,\tau)$ be an admissible pair for $\gfrak$. One has
  \begin{align}\label{eq:bar-ri-pm}
     \overline{r_i(T_{w_X}(E_i))} = \nu_i(-1)^{2\alpha_i(\rho_X^\vee)} \ell_i r_{\tau(i)}(T_{w_X}(E_{\tau(i)}))
  \end{align}
  for all $i\in I\setminus X$.
\end{cor}
\section{The bar involution for quantum symmetric pairs}\label{sec:bar-involution-QSP}
We are now ready to address quantum symmetric pair coideal subalgebras $B_{\bc,\bs}$. We first recall their construction. In Subsection \ref{sec:relations} we give presentations of $B_{\bc,\bs}$ in terms of generators and relations. Combining these relations with the results of Section \ref{sec:skew-ders} will allow us to prove the desired existence of the bar involution for $B_{\bc,\bs}$. 
\subsection{Quantum symmetric pairs}\label{sec:QSP}
Recall that admissible pairs $(X,\tau)$ as given by Definition \ref{def:admissible} parametrize involutive automorphisms of the second kind of $\gfrak$ up to conjugation by automorphisms of $\gfrak$. The involutive automorphism $\theta=\theta(X,\tau):\gfrak\rightarrow \gfrak$ corresponding to the admissible pair $(X,\tau)$ is made up of four ingredients. Let $\omega:\gfrak\rightarrow \gfrak$ denote the Chevalley involution as given in \cite[(1.3.4)]{b-Kac1}, \cite[(2.4)]{a-Kolb12p}. The element $\tau\in \Aut(A,X)$ can be lifted to an automorphism $\tau$ of $\gfrak$ as in \cite[4.19]{a-KW92}. The longest element $w_X\in W_X$ lifts to an element $m_X$ in the Kac-Moody group associated to $\gfrak$ which acts on $\gfrak$ by the adjoint action $\Ad(m_X)$. Finally, for any group homomorphism $x:Q\rightarrow \field^\times$ into the multiplicative group of $\field$ define an automorphism $\Ad(x)$ of $\gfrak$ by $\Ad(x)|_\hfrak=\id_\hfrak$ and $\Ad(x)(v)=x(\alpha)v$ if $\alpha$ is a root of $\gfrak$ and $v$ lies in the corresponding root space. By \cite[Theorem 2.7]{a-Kolb12p} the involutive automorphism $\theta=\theta(X,\tau)$ is given by
\begin{align}\label{eq:theta-def}
  \theta= \Ad(s(X,\tau)) \circ \Ad(m_X) \circ \tau\circ \omega.
\end{align}
Here $s(X,\tau):Q \rightarrow \field^\times$ denotes a group homomorphism given as follows. Let $>$ be a fixed total order on the set $I$. Then $s(X,\tau)$ is determined by
\begin{align}\label{eq:sDef}
  s(X,\tau)(\alpha_j)=\begin{cases}
      1& \mbox{if $j\in X$ or $\tau(j)=j$,}\\
     i^{\alpha_j(2\rho^\vee_X)}& \mbox{if $j\notin X$ and $\tau(j)>
      j$,}\\
     (-i)^{\alpha_j(2\rho^\vee_X)}& \mbox{if $j\notin X$ and $\tau(j)<
      j$,}\\
    \end{cases}
\end{align}
where $i\in \field$ denotes a square-root of $-1$. To shorten notation we will usually write
\begin{align}\label{eq:s(i)}
  s(j)=s(X,\tau)(\alpha_j) \qquad \mbox{for all $j\in I$.}
\end{align}
\begin{rema}
  The results of this paper also hold for $\field=\Q$. The imaginary unit appearing in \eqref{eq:sDef} can be avoided by a slight modification of the group homomorphisms $s(X,\tau)$, see \cite[Remark 5.2]{a-BalaKolb15p}.
\end{rema}
By construction the involutive automorphism $\theta$ given by \eqref{eq:theta-def} satisfies $\theta(h)=-w_X \tau(h)$ for all $h\in \hfrak$. By duality $\theta$ induces a map
\begin{align*}
  \Theta:\hfrak^*\rightarrow \hfrak^*, \quad \Theta(\alpha)=-w_X \tau(\alpha).
\end{align*} 
Define $Q^\Theta=\{\beta\in Q\,|\,\Theta(\beta)=\beta\}$.
For later use we note the following property of the involution $\Theta$. 
\begin{lem}\label{lem:alphai-alphataui}
Let $(X,\tau)$ be an admissible pair. For any $i\in I$ one has
\begin{equation}
  \Theta(\alpha_{\tau(i)})-\alpha_{\tau(i)}=\Theta(\alpha_{i})-\alpha_{i}. \label{eq:alphai-alphataui}
\end{equation}
\end{lem}
\begin{proof}
Note that $w_X(\alpha_i)-\alpha_i\in Q_X$. By Property (2) in Definition \ref{def:admissible} we get
\begin{align*}
  w_X(w_X(\alpha_i)-\alpha_i)=-\tau(w_X(\alpha_i)-\alpha_i).
\end{align*}
This can be rewritten as
$$\alpha_i+\Theta(\alpha_{\tau(i)})=\Theta(\alpha_i)+\alpha_{\tau(i)},$$
which is equivalent to the claim. 
\end{proof}
In \cite[Definition 4.3]{a-Kolb12p} a quantum group analog $\theta_q(X,\tau):\uqg\rightarrow \uqg$ of $\theta$ was constructed. The algebra automorphism $\theta_q=\theta_q(X,\tau)$ was used to define a quantum group analog $B_{\bc,\bs}$ of the universal enveloping algebra $U(\kfrak)$ of the fixed Lie subalgebra $\kfrak=\{x\in \gfrak\,|\,\theta(x)=x\}$. More explicitly, let $U^0_\Theta$ denote the subalgebra of $U^0$ generated by all $K_\beta$ for $\beta\in Q^\Theta$. The quantum symmetric pair coideal subalgebra $B_{\bc,\bs}=B_{\bc,\bs}(X,\tau)$ corresponding to the  admissible pair $(X,\tau)$ is the subalgebra of $\uqg$ generated by $\cM_X$, $U^0_\Theta$, and the elements
\begin{align}\label{eq:Bi-def}
  B_i = F_i + c_i \theta_q(F_i K_i) K_i^{-1} + s_i K_i ^{-1} \qquad \mbox{for all $i\in I\setminus X$}.
\end{align}
Here $\bc=(c_i)_{i\in I\setminus X}$ denotes a family of parameters in the set 
\begin{align}\label{eq:C-def}
  \cC=\{\bc\in (\field(q)^\times)^{I\setminus X}\,|\, c_i=c_{\tau(i)} \mbox{ if } \tau(i)\neq i \mbox{ and } (\alpha_i,\Theta(\alpha_i))= 0\},
\end{align}
and $\bs=(s_i)_{i\in I\setminus X}$ denotes a family of parameters in a set $\cS$ given by
\begin{align}\label{def:setS}
\mathcal{S} &= \{ \mathbf{s}\in \field(q)^{I\setminus X} | s_i\ne 0 \Rightarrow ( i \in I_{ns} \,\, \mathrm{ and }\,\, a_{ji} \in -2\mathbb{N}_0 \,\, \forall j \in I_{ns}\setminus \{i\}) \} 
\end{align}
where 
\begin{align*}
  I_{ns} &= \{ i \in I\setminus X | \tau(i)= i \textrm{ and } a_{ij}=0 \textrm{ for all } j\in X \},  
\end{align*}
see \cite[(5.9)--(5.11)]{a-Kolb12p}.
\begin{rema} 
There is a crucial typo in the definition of $\mathcal{S}$ in \cite{a-Kolb12p}. The correct definition should be \eqref{def:setS}, obtained from \cite[(5.11)]{a-Kolb12p} by replacing $a_{ij}$ by $a_{ji}$. In \cite{a-Kolb12p} the properties of $\mathcal{S}$ are only used in Step 4 in the proof of Lemma 5.11 and in Remark 5.12. Both arguments require the present definition of $\cS$ and not the one given in \cite[(5.11)]{a-Kolb12p}. 
\end{rema}
By construction one has 
\begin{align}\label{eq:thetaq(FiKi)}
  \theta_q(F_i K_i) =-s(\tau(i)) T_{w_X}(E_{\tau(i)})\qquad \mbox{for all $i\in I\setminus X$}
\end{align}
and hence $\theta_q(F_i K_i)\in U^+_{w_x\alpha_{\tau(i)}}$. For $i\in I\setminus X$ and $j\in X$ there exist elements $\cZ_i\in U^+_{w_x\alpha_{\tau(i)}-\alpha_{\tau(i)}} K_i^{-1}K_{\tau(i)}$ 
and $\cW_{ij} \in U^+_{w_x\alpha_{\tau(i)}-\alpha_{\tau(i)}-\alpha_j} K_i^{-1}K_{\tau(i)}$
such that
\begin{align}
  \kow(\theta_q(F_iK_i)K_i^{-1})&= \theta_q(F_iK_i)K_i^{-1} \ot K_i^{-1} + \cZ_i \ot E_{\tau(i)}K_i^{-1} \nonumber \\ 
      &\qquad \qquad+ \cW_{ij} K_j\ot \ad(E_j)(E_{\tau(i)})K_i^{-1} + \dots\label{eq:W-def}
\end{align}
where $\dots$ denotes terms which are not of weight $\alpha_{\tau(i)}$ or $\alpha_{\tau(i)}+\alpha_j$ in the second tensor factor, see \cite[Lemmata 7.2, 7.7]{a-Kolb12p}. By \eqref{eq:ri-def} and \eqref{eq:thetaq(FiKi)} one has
\begin{align}\label{eq:Z-ri}
  \cZ_i = r_{\tau(i)}(\theta_q(F_i K_i)) K_{\tau(i)} K_i^{-1} 
        = -s(\tau(i)) \, r_{\tau(i)}(T_{w_X}(E_{\tau(i)})) K_{\tau(i)} K_i^{-1}.
\end{align}
The element $\cW_{ij}$ can also be expressed in terms of the skew derivation $r_i$. 
\begin{lem}
  Let $i\in I\setminus X$ and $j\in X$. Then the relation
    \begin{align}\label{eq:W-rjri}
      r_jr_{\tau(i)}\big(\theta_q(F_iK_i)\big)K_{\tau(i)}K_i^{-1}= (1-q^{2(\alpha_i,\alpha_j)})\cW_{ij}
    \end{align}
  holds. In particular, if $\tau(i)=i$ and $(\alpha_i,\alpha_j)\neq 0$ then one has
  \begin{align}\label{eq:W-rjri2}
    \cW_{ij}=  \frac{1}{1-q^{2(\alpha_i,\alpha_j)}} \,r_j(\cZ_i).
  \end{align}
\end{lem}
\begin{proof}
  For any $\alpha\in Q^+$ let $\pi_\alpha:\uqg\rightarrow U^+_\alpha U^0$ denote the projection with respect to the direct sum decomposition $\uqg=\oplus_{\alpha,\beta\in Q^+} U^+_\alpha U^0 U^-_{-\beta}$. It follows from \eqref{eq:W-def} that
  \begin{align*}
    (\id \ot \pi_{\alpha_j}&\ot \pi_{\alpha_{\tau(i)}})\circ(\id\ot\kow)\circ \kow (\theta_q(F_iK_i)K_i^{-1})\\
    &=\cW_{ij} K_j \ot \big( (\pi_{\alpha_j}\ot \pi_{\alpha_{\tau(i)}})\circ \kow (\ad(E_j)(E_{\tau(i)})K_i^{-1})\big)\\
    &= \cW_{ij}K_j\ot (1-q^{2(\alpha_i,\alpha_j)}) E_j K_{\tau(i)}K_i^{-1} \ot E_{\tau(i)} K_i^{-1}.
  \end{align*}
  On the other hand, using the definition of $r_j$ in \eqref{eq:ri-def}, one has
  \begin{align*}
    (\id \ot \pi_{\alpha_j}&\ot \pi_{\alpha_{\tau(i)}})\circ(\kow\ot\id)\circ \kow (\theta_q(F_iK_i)K_i^{-1})\\
    &= r_j r_{\tau(i)}\big(\theta_q(F_iK_i)\big)K_j K_{\tau(i)}K_i^{-1}\ot E_j K_{\tau(i)}K_i^{-1} \ot E_{\tau(i)} K_i^{-1}.
  \end{align*}
   By coassociativity of the coproduct the above two expressions coincide and hence one obtains Equation \eqref{eq:W-rjri}.
\end{proof}
Combining Equation \eqref{eq:Z-ri} with Corollary \ref{cor:bar-ri} one sees how the bar involution acts on $\cZ_i$. Recall from \eqref{eq:li-def} that $\ell_i=q^{(\alpha_i,\alpha_i-w_X(\alpha_i)-2\rho_X)}$ and recall the signs $\nu_i\in \{-1,1\}$ from Proposition \ref{prop:sigma-tau-ri}.
\begin{prop}\label{prop:ocZ}
  Let $i\in I\setminus X$. Then one has 
  \begin{align}\label{eq:ocZ}
    \overline{\cZ_i}=\nu_i \ell_i \cZ_{\tau(i)}.
  \end{align}  
\end{prop}
\begin{proof}
  One calculates
  \begin{align*}
    \overline{\cZ_i} &\stackrel{\eqref{eq:Z-ri}}{=} -s(\tau(i)) \overline{r_{\tau(i)}(T_{w_X}(E_{\tau(i)}))} K_{\tau(i)}^{-1} K_i\\
           &\stackrel{\eqref{eq:bar-ri-pm}}{=} -\nu_i s(\tau(i)) (-1)^{2\alpha_i(\rho_X^\vee)} \ell_i r_i(T_{w_X}(E_i))  K_{\tau(i)}^{-1} K_i\\
           & \stackrel{\eqref{eq:sDef}}{=} -\nu_i s(i) \ell_i r_i(T_{w_X}(E_i))  K_{\tau(i)}^{-1} K_i\\
           &\stackrel{\eqref{eq:Z-ri}}{=} \nu_i \ell_i \cZ_{\tau(i)}
  \end{align*}
which completes the proof of \eqref{eq:ocZ}. 
\end{proof}

\subsection{Relations revisited}\label{sec:relations}
The algebra $B_{\bc,\bs}$ can be described explicitly in terms of generators and relations. In the finite case a presentation of the algebra $B_{\bc,\bs}$ was given by G.~Letzter in \cite[Theorem 7.1]{a-Letzter03} and this was extended to the Kac-Moody case in \cite[Section 7]{a-Kolb12p} for small values of $-a_{ij}$. In the present subsection we recall these results and extend them to a larger class of admissible pairs. The resulting explicit presentations will allow us in Subsection \ref{sec:existence} to prove the existence of the bar involution for $B_{\bc,\bs}$. First, however, we need to introduce some notation.

To extend the notation given in \eqref{eq:Bi-def} from $I\setminus X$ to all of $I$ we write $B_i=F_i$ if $i\in X$. For any multi-index $J=(j_1,\dots,j_k)\in I^k$ define $\wght(J)=\sum_{i=1}^k\alpha_{j_i}$ and $F_J=F_{j_1}\dots F_{j_k}$ and $B_J=B_{j_1}\dots B_{j_k}$. Let $\cJ$ be a fixed subset of $\cup_{k\in \N_0}I^k$ such that $\{F_J\,|\,J\in \cJ\}$ is a basis of $U^-$. Then $\{B_J\,|\,J\in \cJ\}$ is a basis of the left (or right) $\cM_X^+ U^0_\Theta$ module $B_{\bc,\bs}$, see \cite[Proposition 6.2]{a-Kolb12p}. 

By \cite[Theorem 7.1]{a-Kolb12p} the algebra $B_{\bc,\bs}$ is generated over $\cM_X^+ U^0_\Theta$ by the elements $\{B_i\,|\,i\in I\}$ subject to the following defining relations
\begin{align}
     &K_\beta B_i = q^{-(\beta,\alpha_i)} B_i K_\beta && \mbox{for all $\beta\in Q^\Theta$, $i\in I$,} \label{eq:rel1}\\
     & E_i B_j - B_j E_i = \delta_{ij} \frac{K_i - K_i^{-1}}{q_i-q_i^{-1}} && \mbox{for all $i\in X$, $j\in I $,}\label{eq:rel2}\\
     &F_{ij}(B_i,B_j)=C_{ij}(\uc) &&\mbox{for all $i,j\in I$, $i\neq j$.} \label{eq:rel3}
  \end{align}
Here $C_{ij}(\bc)\in\sum_{\wght(J)< \lambda_{ij}} \cM^+_X U^0_\Theta B_J$ with $\lambda_{ij}=(1-a_{ij})\alpha_i+\alpha_j$ can be explicitly determined. In \cite[7.2, 7.3]{a-Kolb12p} this was done for all $i,j$ with $-2\le a_{ij}\le 0$. In the present section we generalize these formulas. By \cite[Theorem 7.3]{a-Kolb12p} one has 
\begin{align*}
  C_{ij}(\bc)=0 \quad \mbox{unless $i=\tau(i)$ or $j=\tau(i)$.} 
\end{align*}
In the case $j=\tau(i)$ we can give an explicit 
formula for $C_{ij}(\bc)$ which holds for any $a_{ij}$. To simplify notation recall the $q$-shifted factorial
\begin{align}\label{eq:phi-def}
  (x;x)_n= \prod_{k=1}^{n}(1-x^k)\qquad \mbox{for any $n\in \N$.} 
\end{align} 
The proof of the following theorem is given in Subsection \ref{sec:proof-relations}.
\begin{thm}\label{thm:Citaui}
  Assume that $i\in I\setminus X$ satisfies $\tau(i)\neq i$ and let $m=1-a_{i\tau(i)}$. Then one has
  \begin{align*}
    C_{i\tau(i)}(\bc)&=\frac{-1}{(q_i-q_i^{-1})^2}\Big(q_i^{-m} (q_i^2;q_i^2)_m B_i^{m-1} c_i\cZ_i
      + q_i (q_i^{-2};q_i^{-2})_m B_i^{m-1} c_{\tau(i)}\cZ_{\tau(i)}\Big).
  \end{align*}
\end{thm}
It remains to consider the case $\tau(i)=i$. In this case the calculations get more involved and so far no general closed formula for $C_{ij}(\bc)$ has been found. The following two theorems provide $C_{ij}(\bc)$ for small values of $-a_{ij}$. The theorems are contained in \cite[Theorems 7.4, 7.8]{a-Kolb12p} apart from Case 4 in Theorem \ref{thm:relsBc} which is obtained by the same methods as the other cases and therefore left to the reader. 
\begin{thm}\label{thm:relsBc}
  Assume that $i,j\in I\setminus X$ and $\tau(i)=i$. Then the elements $C_{ij}(\uc)$ are given by the following formulas.
  
\noindent {\bf Case 1:} $a_{ij}=0$.
  \begin{align}\label{eq:Cij0}
    C_{ij}(\uc)= 0.
  \end{align}
\noindent {\bf Case 2:} $a_{ij}=-1$.
  \begin{align}
    C_{ij}&(\uc)= q_i c_i \cZ_i B_j \label{eq:Cij1}
  \end{align}  
\noindent {\bf Case 3:} $a_{ij}=-2$.
  \begin{align}\label{eq:Cij2}
    C_{ij}(\uc)=&  [2]_{q_i}^2 q_i c_i \cZ_i(B_i B_j - B_j B_i). 
  \end{align}        
\noindent {\bf Case 4:} $a_{ij}=-3$.
  \begin{align}\label{eq:Cij3}
    C_{ij}(\uc)=& - ([3]_{q_i}^2 + 1) (B_i^2 B_j + B_j B_i^2) q_i c_i \cZ_i \\
      & +  [2]_{q_i}([2]_{q_i}[4]_{q_i} + q_i^2 + q_i^{-2})B_i B_j B_i q_i c_i \cZ_i
        - [3]_{q_i}^2 B_j (q_i c_i \cZ_i)^2. \nonumber
  \end{align}  
\end{thm}
\begin{thm}\label{thm:relsBc2}
  Assume that $i\in I\setminus X$ and $j\in X$ and $\tau(i)=i$. Then the elements $C_{ij}(\uc)$ are given by the following formulas.

\noindent {\bf Case 1:} $a_{ij}=0$.
  \begin{align}\label{eq:Cij0X}
    C_{ij}(\uc)= 0.
  \end{align}
\noindent {\bf Case 2:} $a_{ij}=-1$.
  \begin{align}
    C_{ij}&(\uc)=  \frac{1}{q_i-q_i^{-1}}\big(q_i^2 B_j c_i\cZ_i - c_i\cZ_i B_j\big)
     +\frac{q_i+q_i^{-1}}{q_j-q_j^{-1}}c_i\cW_{ij} K_j. \label{eq:Cij1X}
  \end{align}  
\noindent {\bf Case 3:} $a_{ij}=-2$.
  \begin{align}\label{eq:Cij2X}
    C_{ij}(\uc)=\frac{1}{q_i-q_i^{-1}}\bigg[& q_i^2 \Big([3]_{q_i} B_i B_j - (q_i^2 + 2)B_j B_i\Big)c_i\cZ_i \\
     &\qquad-c_i\cZ_i\Big((2+q_i^{-2})B_iB_j - [3]_{q_i} B_j B_i \Big)\bigg]\nonumber\\ 
     -&\frac{q_i-q_i^{-1}}{q_j-q_j^{-1}}(q_i+q_i^{-1})^2[3]_{q_i} B_i c_i \cW_{ij} K_j.\nonumber
  \end{align}        
\end{thm}
By \eqref{eq:ri-Fcomm} we have for $i\in I\setminus X$ and $j\in X$ the relation
\begin{align*}
  \cZ_i B_j - B_j \cZ_i =\frac{r_j(\cZ_i) K_j - K_j^{-1} {}_j r(\cZ_i)}{q_j-q_j^{-1}}.
\end{align*}
Applying this to \eqref{eq:Cij1X} and \eqref{eq:Cij2X} and using \eqref{eq:W-rjri2} the formulas in Theorem \ref{thm:relsBc2} can be written in a way similar to those in Theorem \ref{thm:relsBc}. This gives a unified presentation of Theorems \ref{thm:relsBc} and \ref{thm:relsBc2} in the case
$-2\le a_{ij}\le 0$.
\begin{thm}\label{thm:relsBc3}
  Assume that $i\in I\setminus X$ and $j\in I\setminus\{i\}$ and $\tau(i)=i$. Then the elements $C_{ij}(\uc)$ are  
  given by the following formulas.

\noindent {\bf Case 1:} $a_{ij}=0$.
  \begin{align}\label{eq:Cij0X-v2}
    C_{ij}(\uc)= 0.
  \end{align}
\noindent {\bf Case 2:} $a_{ij}=-1$.
  \begin{align}\label{eq:Cij1X-v2}
    C_{ij}(\uc)=  q_i B_j c_i\cZ_i + \frac{q_i^2c_i r_j(\cZ_i)K_j + q_j^2 q_i^{-2}c_i \,{}_j r(\cZ_i)K_j^{-1}}{(q_i-q_i^{-1})(q_j-q_j^{-1})}.
  \end{align}  
\noindent {\bf Case 3:} $a_{ij}=-2$.
  \begin{align}\label{eq:Cij2X-v2}
    C_{ij}(\uc)=& q_i [2]_{q_i}^2 (B_i B_j - B_j B_i) c_i \cZ_i\\ 
       &\quad-\frac{q_i^4 [2]_{q_i}}{q_j-q_j^{-1}} B_i c_i r_j(\cZ_i) K_j 
       + q_j^2 \frac{q_i^{-6} [2]_{q_i}}{q_j-q_j^{-1}} B_i c_i \,{}_j r(\cZ_i) K_j^{-1}.\nonumber
  \end{align}        
\end{thm}
\begin{proof}
  As explained above, for $j\in X$ Formulas \eqref{eq:Cij1X-v2} and \eqref{eq:Cij2X-v2} follow from Theorem \ref{thm:relsBc2}. For $j\in I\setminus X$ one has $r_j(\cZ_i)={}_j r(\cZ_i)=0$ and hence the above formulas coincide with those in Theorem \ref{thm:relsBc}. 
\end{proof}
\begin{rema}
  It is desirable to obtain a closed formula for $C_{ij}(\bc)$ in the setting of the above theorem for general values of $a_{ij}$. In this case one has 
  \begin{align*}
    C_{ij}(\bc)=\sum_{l=1}^{[m/2]} \sum_{k=0}^{m-2l} B_i^k B_j B_i^{m-k-2l} b_{k,l} + \sum_{l=1}^{[m/2]} B_i^{m-2l} d_l
  \end{align*}
for uniquely determined coefficients $b_{k,l}$, $d_l\in  \cM_X^+ U^0_\Theta$ and $m=1-a_{ij}$. Theorem \ref{thm:relsBc}, Case 4, and Theorem \ref{thm:relsBc3}, Cases 2 and 3, indicate that the coefficients $b_{k,l}$ and $d_l$ are fairly involved. The fact that so many linearly independent monomials $B_J$ enter the expression for $C_{ij}(\bc)$ makes this case harder than the case $\tau(i)=j$ treated in Theorem \ref{thm:Citaui}.
\end{rema}
\subsection{Existence of the bar involution.}\label{sec:existence}
  In Subsection \ref{sec:relations} the relations for $B_{\bc,\bs}$ are given explicitly for Cartan matrices $(a_{ij})$ and admissible pairs $(X,\tau)$ with the following properties
  \begin{enumerate}
     \item[(i)]  If $i\in I\setminus X$ with $\tau(i)=i$ and $j\in X$ then $a_{ij}\in \{0,-1,-2\}$.
     \item[(ii)]  If $i\in I\setminus X$ with $\tau(i)=i$ and $i\neq j\in I\setminus X$ then $a_{ij}\in \{0,-1,-2,-3\}$.
  \end{enumerate}
For this reason, in the present subsection, we also restrict to Cartan matrices and admissible pairs which satisfy the two properties above. In particular, this includes all finite types. In this fairly general setting we are now in a position to prove the main result of the present paper.
\begin{thm}\label{thm:bar-involution}
  The following statements are equivalent.
  \begin{enumerate}
    \item There exists a $\field$-algebra automorphism $\overline{\phantom{B}}:B_{\bc,\bs}\rightarrow 
          B_{\bc,\bs}$, $x\mapsto \overline{x}$ with the following properties:
        \begin{enumerate}
           \item  Restricted to $\cM_X U_\Theta^0{}$ the map $\overline{\phantom{B}}$ coincides with the bar involution of $\uqg$. In particular $\overline{q}=q^{-1}$.
           \item  $\overline{B_i}=B_i$ for all $i\in I\setminus X$.
        \end{enumerate}  
    \item The relation
       \begin{align}\label{eq:ocZi}
         \overline{c_i \cZ_i} = q^{(\alpha_i,\alpha_{\tau(i)})} c_{\tau(i)} \cZ_{\tau(i)}
       \end{align}
         holds for all $i\in I\setminus X$ for which $\tau(i)\neq i$ or for which there exists $j\in I\setminus \{i\}$ such that $a_{ij}\neq 0$. 
  \end{enumerate}
\end{thm}
\begin{proof}
 A $\field$-algebra automorphism as described in (1) exists if and only if relations \eqref{eq:rel1}, \eqref{eq:rel2}, and \eqref{eq:rel3} are preserved under the rule $B_i \mapsto B_i$, $E_j\mapsto E_j$,  $K_\beta \mapsto K_\beta^{-1}$, $q\mapsto q^{-1}$ for $i\in I$, $j\in X$, $\beta\in Q^\Theta$. One verifies that relations \eqref{eq:rel1} and \eqref{eq:rel2} are preserved under this rule, and the same applies for the left hand side of Relation \eqref{eq:rel3}. To deal with the right hand side, recall that $\lambda_{ij}=(1-a_{ij})\alpha_i+\alpha_j$ and write
 \begin{align*}
   C_{ij}(\bc) = \sum_{\mathrm{wt}(J)< \lambda_{ij}} B_J a_{ij;J} \qquad \mbox{for some $a_{ij;J}\in \cM_X^+ U^0_\Theta$}
 \end{align*}
where $B_J=B_{j_1}\dots B_{j_k}$ for $J=(j_1,\dots,j_k)$ as explained at the beginning of Subsection \ref{sec:relations}. Observe that the set $\{B_J\,|\,\wght(J)<\lambda_{ij}\}$ is linearly independent in the right $\cM_X^+U^0_\Theta$-module $B_{\bc,\bs}$ and hence the coefficients $a_{ij;J}$ are uniquely determined. Therefore a $\field$-algebra automorphism as described in (1) exists if and only if the bar involution of $\uqg$ satisfies
\begin{align}\label{eq:bar-condition}
  \overline{a_{ij;J}}=a_{ij;J}
\end{align}
for all $i\in I\setminus X$, $j\in I\setminus\{i\}$. This condition only concerns cases where $C_{ij}(\bc)\neq 0$, that is $\tau(i)=j\neq i$ or ($\tau(i)=i$ and $a_{ij}\neq 0$ for some $j\neq i$). If $\tau(i)=j$ then Theorem \ref{thm:Citaui} together with $\overline{(q_i^2;q_i^2)_m}=(q_i^{-2};q_i^{-2})_m$ implies that \eqref{eq:bar-condition} is equivalent to \eqref{eq:ocZi}. If $\tau(i)=i$ and $a_{ij}\neq 0$ for some $j\neq i$ then Theorem \ref{thm:relsBc3} and Theorem \ref{thm:relsBc}.Case 4 imply that \eqref{eq:bar-condition} is equivalent to \eqref{eq:ocZi}. For illustration we make this explicit in the case $a_{ij}=-2$ of Theorem \ref{thm:relsBc3}. In this case one has
\begin{align}
  a_{ij;(i,j)}&= -a_{ij;(j,i)}= q_i [2]_{q_i}^2 c_i \cZ_i\nonumber\\
  a_{ij;(i)}&= \frac{[2]_{q_i}}{q_j-q_j^{-1}} \big(q_j^2 q_i^{-6} c_i\, {}_j r(\cZ_i) K_j^{-1} - 
      q_i^4 c_i r_j(\cZ_i) K_j\big).\label{eq:aiji}
\end{align}
One obtains $\overline{a_{ij;(i,j)}}=q_i^{-1}[2]_{q_i}^2 \overline{c_i\cZ_i}$. Hence \eqref{eq:bar-condition} holds for $J=(i,j)$ if and only if $\overline{c_i\cZ_i}=q_i^2 c_i\cZ_i$. 
It remains to show that \eqref{eq:ocZi} also implies that $\overline{a_{ij;(i)}}=a_{ij;(i)}$. To this end observe that by \eqref{eq:bar-ri} one has
\begin{align*}
  \overline{r_j(\cZ_i)} = q_j^2 q_i^{2a_{ij}} {}_j r(\overline{\cZ_i}).
\end{align*}
The above relation and \eqref{eq:aiji} allow us to calculate
\begin{align*}
\overline{a_{ij;(i)}}&= \frac{[2]_{q_i}}{q_j-q_j^{-1}} \big(q_i^{-4} \overline{c_i r_j(\cZ_i)} K_j^{-1} - q_j^{-2} q_i^6 \overline{c_i\, {}_j r(\cZ_i)}K_j\big)\\
    &=\frac{[2]_{q_i}}{q_j-q_j^{-1}} \big(q_j^2 q_i^{-6}{}_j r(q_i^{-2}\overline{c_i\cZ_i}) K_j^{-1} -  q_i^4 r_j(q_i^{-2}\overline{c_i\cZ_i})K_j\big).
\end{align*}
Hence \eqref{eq:ocZi} implies \eqref{eq:bar-condition} also in the case $J=(i)$.
\end{proof}
\begin{rema}
  The condition ($\tau(i)\neq i$ or there exists $j\in I\setminus \{i\}$ such that $a_{ij}\neq 0$) means that the Cartan matrix $A=(a_{ij})_{i,j\in I}$ does not contain a component of type $A_1$ which is fixed under $\tau$. If, however, such a component exists and is indexed by $i$, then the existence of a bar involution does not depend on the value of the corresponding parameter $c_i$.
\end{rema}
Theorem \ref{thm:bar-involution} implies that for any admissible pair the coefficients $\bc\in \cC$ can be chosen such that a bar involution exists on $B_{\bc,\bs}$. The following corollary treats the case that Conjecture \ref{conj:sigma-tau-ri} holds for $\gfrak$ and $(X,\tau)$. The case that $\nu_i=-1$ for some $i\in I\setminus X$ is discussed in Remark \ref{rem:nu-1}. 
Define
\begin{align}\label{eq:field0}
  \field_0(q)=\{\lambda\in \field(q)\,|\,\overline{\lambda}=\lambda\}
\end{align}
and observe that this is a subfield of $\field(q)$. 
\begin{cor}\label{cor:nu=1}
  Assume that $\nu_i=1$ for all $i\in I\setminus X$. The following are equivalent:
  \begin{enumerate}
    \item There exists a $\field$-algebra automorphism $\overline{\phantom{B}}:B_{\bc,\bs}\rightarrow 
          B_{\bc,\bs}$ satisfying properties (a) and (b) in Theorem \ref{thm:bar-involution}.(1).
    \item The parameters $\bc=(c_i)_{i\in I\setminus X}\in \cC$ satisfy the following two conditions:
      \begin{enumerate}
        \item If ($\tau(i)=i$ and $a_{ij}\neq 0$ for some $j\in I\setminus \{i\}$) or $(\alpha_i,\Theta(\alpha_i))=0$ then there exists $\lambda_i\in\field_0(q)^\times$ such that
          \begin{align}\label{eq:cond1}
            c_i=c_{\tau(i)}= \lambda_i q^{(\alpha_i,\Theta(\alpha_i)-2\rho_X)/2}.   
          \end{align}
        \item If $\tau(i)\neq i$ and $(\alpha_i,\Theta(\alpha_i))\neq 0$ then 
          \begin{align}\label{eq:cond2}
             c_{\tau(i)}&= q^{(\alpha_i,\Theta(\alpha_i)-2\rho_X)}\overline{c_i}.
          \end{align}
      \end{enumerate}
  \end{enumerate}    
\end{cor}      
\begin{proof}
  If $\nu_i=1$ then by Proposition \ref{prop:ocZ} Relation \eqref{eq:ocZi} is equivalent to
  $\overline{c_i}=q^{(\alpha_i,\alpha_{\tau(i)})} \ell_i^{-1} c_{\tau(i)}$. Using Lemma \ref{lem:alphai-alphataui} this can be rewritten as
  \begin{align}\label{eq:cond3}
    \overline{c_i} = q^{(\alpha_i,w_X(\alpha_i)-\alpha_i+\alpha_{\tau(i)}+2\rho_X)}c_{\tau(i)}
    =q^{(\alpha_i, 2\rho_X-\Theta(\alpha_i))} c_{\tau(i)}
  \end{align}  
  and has to hold for all $i\in I\setminus X$ which satisfy one of the hypotheses in Conditions (2)(a) or (2)(b) of the Corollary. If ($\tau(i)=i$ and $a_{ij}\neq 0$ for some $j\in I\setminus X$) or ($\tau(i)\neq i$ and $(\alpha_i,\Theta(\alpha_i))\neq 0$) then Equation \eqref{eq:cond3} is equivalent to Equations \eqref{eq:cond1} and \eqref{eq:cond2}, respectively. 
  
It remains to consider the case that $\tau(i)\neq i$ and $(\alpha_i,\Theta(\alpha_i))= 0$. In this case $\Theta(\alpha_i)=-\alpha_{\tau(i)}$, see \cite[Lemma 5.3]{a-Kolb12p}, and hence $w_X(\alpha_i)=\alpha_i$ and $(\alpha_i,2\rho_X)=0$. In this case \eqref{eq:cond1} is equivalent to $c_i=c_{\tau(i)}=\lambda_i\in \field_0(q)^\times$. The relation $c_i=c_{\tau(i)}\in \field(q)^\times$ is necessary because of the definition of $\cC$, see \eqref{eq:C-def}. Under this assumption the relation $c_i=c_{\tau(i)}\in \field_0(q)^\times$ given by \eqref{eq:cond1} is indeed equivalent to the relation $\overline{c_i}=c_{\tau(i)}$ given by \eqref{eq:cond3}.
\end{proof}
\begin{rema}\label{rem:q-exponent-even}
It can be shown that the exponent $(\alpha_i,\Theta(\alpha_i)-2\rho_X)/2$ in \eqref{eq:cond1} is always an integer. This, together with Corollary \ref{cor:nu=1}, shows that if $\nu_i=1$ then the Equations \eqref{eq:ocZi} for $i\in I\setminus X$ always have solutions $c_i\in\mathbb{K}(q)$. Moreover, these solutions can be chosen to be powers of $q$ and hence specialize to $1$ at $q=1$. 

To prove that the exponent in \eqref{eq:cond1} is always an integer, assume that we are in case (2)(a) of the corollary, that is $i\in I\setminus X$ satisfies ($\tau(i)=i$ and $a_{ij}\neq 0$ for some $j\in I\setminus \{i\}$) or $(\alpha_i,\Theta(\alpha_i))=0$. If $\left(\alpha_i,\Theta(\alpha_i)\right)=0$ and $\tau(i)\ne i$, then by \cite[Lemma 5.3]{a-Kolb12p} we have $(\alpha_i, 2\rho_X)=0$, and hence $(\alpha_i,\Theta(\alpha_i)-2\rho_X)/2=0$. 

It remains to consider the case that $\tau(i)=i$. Only considering a connected component of the Dynkin diagram of $\gfrak_X$ we may assume that $X$ is connected and that $(X,\tau)$ satisfies  all conditions in Definition \ref{def:admissible} of an admissible pair except possibly condition (3).
We want to show that $(\alpha_i,w_X(\alpha_i)+2\rho_X)$ is even. 
Fix a reduced expression of the longest word $w_X=s_{i_1}s_{i_2}\ldots s_{i_k}$ in $W_X$ and recall that
$\{\beta_l=s_{i_1}s_{i_2}\ldots s_{i_{l-1}}(\alpha_{i_l})\,|\,l=1,\ldots k\}$ are the positive roots of the root system $\Phi_X$. 
We calculate 
\begin{align*}
(\alpha_i,w_X(\alpha_i)+2\rho_X)&= \Big(\alpha_i,s_{i_1}s_{i_2}\ldots s_{i_k}(\alpha_i)+\sum_{l=1}^k\beta_l\Big)\\
&= (\alpha_i,\alpha_i)+\sum_{l=1}^k(1-\alpha_i(h_{i_l}))(\alpha_i,\beta_l).
\end{align*}
The inner product $(\alpha_i,\alpha_i)=2\epsilon_i$ is always even. By a case by case investigation for each finite simple root system $\Phi_X$ one can verify that the number $\sum_{l=1}^n(1-\alpha_i(h_{i_l}))(\alpha_i,\beta_l)$ is also even. 
\end{rema}
\begin{rema}\label{rem:nu-1}
  For general $\nu_i$, $i\in I\setminus X$ and $\bc=(c_i)_{i\in I\setminus X}\in \cC$ define 
  \begin{align*}
    d_i=\begin{cases} c_i & \mbox{if $\nu_i=1$,}\\
                   (q-q^{-1})c_i & \mbox{if $\nu_i=-1$.}
        \end{cases}
  \end{align*}
Then there exists a bar involution as in Theorem \ref{thm:bar-involution}.(1) on $B_{\mathbf{d},\bs}$ for $\mathbf{d}=(d_i)_{i\in I\setminus X}$ if and only if $\bc$ satisfies Condition (2) of Corollary \ref{cor:nu=1}. Indeed, in this case the condition $\overline{d_i\cZ_i}=q^{(\alpha_i,\alpha_{\tau(i)})} d_{\tau(i)} \cZ_{\tau(i)}$ is equivalent to $\overline{d_i}=\nu_i q^{(\alpha_i,\alpha_{\tau(i)})}\ell_i^{-1} d_{\tau(i)}$ which in turn is equivalent to the relation $\overline{c_i}= q^{(\alpha_i,\alpha_{\tau(i)})}\ell_i^{-1} c_{\tau(i)}$. This shows that the bar involution also exists if $\nu_i=-1$ for some $i\in I\setminus X$.
\end{rema}
\begin{eg}\label{eg:AIII-IV}
  Consider the case that $\gfrak=\slfrak_{n+1}(\field)$ and $(X,\tau)$ is of type $AIII/IV$ in Araki's table \cite[p.~32]{a-Araki62}. We use the standard enumeration of simple roots $I=\{1,\dots,n\}$ and $\tau$ is the nontrivial diagram automorphism. There are two cases.
 
\noindent {\bf Case I.} $X=\{r+1,r+2,\dots,n-r\}$ for some $r\le n/2$. To determine the exponents in \eqref{eq:cond1} and \eqref{eq:cond2} one calculates
\begin{align*}
  (\alpha_r, \Theta(\alpha_r)-2\rho_X)=n-2r+1.
\end{align*}  
Hence, by Corollary \ref{cor:nu=1}, a bar involution as in Theorem \ref{thm:bar-involution}.(1) exists on $B_{\bc,\bs}$ if and only if
\begin{align*}
  &c_i=c_{n+1-i}\in \field_0(q)^\times \qquad \mbox{for $i=1,2,\dots,r-1$},\\
  &c_r\in \field(q)^\times\mbox{ arbitrary, and } c_{n+1-r} = q^{n-2r+1}\overline{c_r}. 
\end{align*}
\noindent {\bf Case II.} $X=\emptyset$ and $n$ odd. In this case, by Corollary \ref{cor:nu=1}, a bar involution as in Theorem \ref{thm:bar-involution}.(1) exists on $B_{\bc,\bs}$ if and only if
\begin{align*}
  &c_i=c_{n+1-i}\in \field_0(q)^\times \qquad \mbox{for $i=1,2,\dots,(n-1)/2$}\\
  &\mbox{and }\quad c_{(n+1)/2}\in q^{-1}\field_0(q)^\times. 
\end{align*}
\end{eg}
\begin{rema}
  The papers \cite{a-BaoWang13p} and \cite{a-EhrigStroppel13p} consider a bar involution for any quantum symmetric pair coideal subalgebra $B_{\bc,\bs}$ of type AIII/IV with $X=\emptyset$. Up to conventions this bar involution is precisely the bar involution obtained here for a special choice of $\bc$ and $\bs$. We make this precise in the case where $n=2r+1$ is odd, the case where $n$ is even being analogous.
  
In \cite{a-BaoWang13p} Bao and Wang consider a Hopf algebra $\cU$ over $\Q(q)$ with generators $E_i, F_i, K_i^{\pm 1}$ for $i\in I$ such that $\field(q)\ot_{\Q(q)} \cU$ is isomorphic over $\field$ to  $U_q(\slfrak_n)$ as defined here under the map
\begin{align*}
  1\ot E_i \mapsto F_i, \qquad 1\ot F_i \mapsto E_i, \qquad 1\ot K_i\mapsto K_i, \qquad q\mapsto q^{-1}.
\end{align*}  
In our notation, they consider the subalgebra $\cU^\iota$ of $U_q(\slfrak_n)$ generated by the elements
\begin{align}\label{eq:BW-generators}
  \begin{aligned}
  e_{\alpha_i}&=F_i + E_i K_{\tau(i)}^{-1},& t&= F_r + q^{-1} E_r K_r^{-1} + K_r^{-1}.\\
  f_{\alpha_i}&=F_{\tau(i)} + E_i K_{\tau(i)}^{-1}, & k_{\alpha_i}^{\pm 1}&=(K_i K_{\tau(i)}^{-1})^{\pm 1}.
  \end{aligned}
\end{align}
for $i=1, \dots,r-1$ with $\tau(i)=n-i+1$, see \cite[Proposition 2.2]{a-BaoWang13p}. Hence $U^\iota=B_{\bc,\bs}$ where $\bc=(c_i)$ and $\bs=(s_i)$ are given by
\begin{align}\label{eq:ci-def}
  c_i&=\begin{cases} -q^{-1} & \mbox{if $i=r$,}\\ -1 & \mbox{else,}  \end{cases} &
  s_i&=\begin{cases} 1 & \mbox{if $i=r$,}\\ 0 & \mbox{else.}  \end{cases}
\end{align}
The parameters $\bc$ do indeed satisfy the conditions in Case II of Example \ref{eg:AIII-IV}. Moreover, the bar involution on $U^\iota$ given by \cite[Lemma 2.1.(3)]{a-BaoWang13p} coincides with the bar involution on $B_{\bc,\bs}$ from Corollary \ref{cor:nu=1}. 

Similarly, Ehrig and Stroppel consider an (infinite version) of $U_q(\glfrak_n)$ denoted by $\cU^{\shalfnote}$ which has the same coproduct as $U_q(\slfrak_n)$ in our conventions. They define a coideal subalgebra $\cH^{\shalfnote}$ of $\cU^{\shalfnote}$ in \cite[Definition 6.12]{a-EhrigStroppel13p}. The generators $B_i$ of $\cH^{\shalfnote}$ given in \cite[Definition 6.9]{a-EhrigStroppel13p} are identical to the generators $B_i$ in \eqref{eq:Bi-def} for $(c_i)_{i\in I}$ given by \eqref{eq:ci-def} and $s_i=0$ for all $i\in I$.
\end{rema}
\subsection{Equivalence of quantum symmetric pairs with bar involution}\label{sec:equivalence}
Recall from \cite[Definition 9.1]{a-Kolb12p} that we call two coideal subalgebras $B$ and $B'$ of $\uqg$ equivalent if there exists a Hopf algebra automorphism $\varphi:\uqg\rightarrow \uqg$ such that $\varphi(B)=B'$. It is natural to choose the parameters $\bc, \bs$ as simple as possible for $B_{\bc,\bs}$ in a given equivalence class. This was discussed in \cite[Section 9]{a-Kolb12p}, however, the choice of parameters given there yields coideal subalgebras which do not necessarily allow a bar involution as in Theorem \ref{thm:bar-involution}.(1). In the present subsection we reformulate \cite[Proposition 9.2]{a-Kolb12p} taking the bar involution into account.

To obtain sufficiently many automorphisms it makes sense to work with a field extension $\field^{1/2}(q)$ of $\field(q)$ which is closed under quadratic extensions. In \cite[Section 9]{a-Kolb12p} we used the algebraic closure, but only quadratic extensions were actually needed. The bar involution extends to $\field^{1/2}(q)$. As in \eqref{eq:field0} let $\field_0^{1/2}(q)$ denote all elements of $\field^{1/2}(q)$ which are invariant under the bar involution. Moreover, define $\cC^{1/2}$ by \eqref{eq:C-def} with $\field(q)$ replaced by $\field^{1/2}(q)$.

Assume now that $\uqg$ is defined over $\field^{1/2}(q)$. Let $\Htil=\Hom(Q,\field^{1/2}(q)^\times)$ denote the set of group homomorphisms from the root lattice $Q$ into the multiplicative group of $\field^{1/2}(q)$. For any $x\in \Htil$ define a Hopf algebra automorphism $\Ad(x)$ of $\uqg$ by
\begin{align*}
  \Ad(x)(u) = x(\beta)u \qquad \mbox{for all $u\in \uqg_\beta$ and $\beta \in Q$.}
\end{align*}
Up to diagram automorphisms all Hopf algebra automorphisms of $\uqg$ are of the form $\Ad(x)$ for some $x\in \Htil$, see \cite{a-Twietmeyer92}, \cite[Theorem 3.2]{a-Kolb12p}. For any two coideal subalgebras $B$ and $B'$ of $\uqg$ we write $B \stackrel{\Htil}{\sim} B'$ if there exists $x\in \Htil$ such that $\Ad(x)(B)=B'$.
In modification of \cite[(9.1)]{a-Kolb12p} define
\begin{align*}
  \cD^{1/2}=\big\{\bd\in(&\field^{1/2}(q)^\times)^{I\setminus X}\,\big|\\
      &d_i = q^{(\alpha_i,\Theta(\alpha_i)-2\rho_X)/2}\quad \mbox{if $\tau(i)=i$ or $(\alpha_i,\Theta(\alpha_i))=0$,}\\
    &d_{\tau(i)} = q^{(\alpha_i,\Theta(\alpha_i)-2\rho_X)}\overline{d_i}\quad \mbox{else}  \big\}
\end{align*}
and define $\cS^{1/2}$ by \eqref{def:setS} with $\field(q)$ replaced by $\field^{1/2}(q)$. For $\bd=(d_i)_{i\in I\setminus X}\in \cD^{1/2}$ and  $\bd'=(d'_i)_{i\in I\setminus X}\in \cD^{1/2}$ define
\begin{align*}
  \bd \stackrel{\cD}{\sim} \bd' \quad \Longleftrightarrow \quad d_i'd_i^{-1}\in \field_0^{1/2}(q) \mbox{ for all $i\in I\setminus X$ with $\tau(i)\neq i$}. 
\end{align*}
An equivalence relation $\stackrel{\cS}{\sim}$ on $\cS^{1/2}$ was defined in \cite[Section 9.1]{a-Kolb12p} by
\begin{align*}
 \bs \stackrel{\cS}{\sim} \bs' \quad \Longleftrightarrow \quad s_i=\pm s_i' \mbox{ for all $i\in I_{ns}$.}
\end{align*}
The second and the third statement of the following proposition are obtained by the same methods as \cite[Proposition 9.2]{a-Kolb12p} and hence the proof is omitted. The first statement of the proposition is an immediate consequence of Corollary \ref{cor:nu=1}.
\begin{prop}
  Assume that $\nu_i=1$ for all $i\in I\setminus X$. 
  \begin{enumerate}
    \item Let $\bd\in \cD^{1/2}$ and $\bs\in \cS^{1/2}$. Then there exists a $\field$-algebra automorphism $\overline{\phantom{B}}:B_{\bd,\bs}\rightarrow B_{\bd,\bs}$ satisfying properties (a) and (b) in Theorem \ref{thm:bar-involution}.(1).
    \item Let $\bc\in \cC^{1/2}$ and $\bs\in \cS^{1/2}$ be such that $B_{\bc,\bs}$ has a bar involution as in Theorem \ref{thm:bar-involution}.(1). Then there exists $\bd\in \cD^{1/2}$ and $\bs'\in \cS^{1/2}$ such that 
    $B_{\bc,\bs}\stackrel{\Htil}{\sim}B_{\bd,\bs'}$.
    \item Let $\bd, \bd'\in \cD^{1/2}$ and $\bs,\bs'\in \cS^{1/2}$. Then $B_{\bd,\bs}\stackrel{\Htil}{\sim}B_{\bd',\bs'}$ if and only if $\bd \stackrel{\cD}{\sim} \bd'$ and $\bs \stackrel{\cS}{\sim}\bs'$.
  \end{enumerate}
\end{prop}
\begin{eg}
  Consider again the case that $\gfrak=\slfrak_{n+1}(\field)$ and $(X,\tau)$ is of type $AIII/IV$ and keep the notation from Example \ref{eg:AIII-IV}.
   
\noindent{\bf Case I.} $X=\{r+1,r+2,\dots,n-r\}$ for some $r\le n/2$. In this case any quantum symmetric pair coideal subalgebra for $(X,\tau)$ which allows a bar involution as in Theorem \ref{thm:bar-involution}.(1) is equivalent to $B_{\bc,\bs}$ where
\begin{align*}
  &c_i=c_{n+1-i}=1 \quad \mbox{for $i=1,\dots,r-1$},\\
  &c_r\quad \mbox{arbitrary, and $c_{n+1-r}=q^{n-2r+1}\overline{c_r}.$}
\end{align*} 
In particular, $c_r$ parametrizes a family of non-equivalent quantum symmetric pairs. The parameters $\bs=(s_i)_{i\in I\setminus X}$ are all zero.

\noindent{\bf Case II.} $X=\emptyset$ and $n$ is odd.
In this case any quantum symmetric pair coideal subalgebra for $(X,\tau)$ which allows a bar involution as in Theorem \ref{thm:bar-involution}.(1) is equivalent to $B_{\bc,\bs}$ where
\begin{align*}
  c_i=c_{n+1-i}=1 \quad \mbox{for $i=1,\dots,(n-1)/2$} \quad \mbox{and } c_{(n+1)/2}=q^{-1}.
\end{align*} 
The parameters $\bc$ are uniquely determined by the requirement that $\bc\in \cD^{1/2}$. One still has a one-parameter family of non-equivalent quantum symmetric pairs for $(X,\tau)$, however, in this case, the parameter comes from the freedom to choose $\bs$.
\end{eg}
\subsection{Proof of Theorem \ref{thm:Citaui}}\label{sec:proof-relations}
We recall the method to determine $C_{ij}(\bc)$ which was originally devised in \cite[Section 7]{a-Letzter03}. Consider the direct sum decomposition 
\begin{align*}
  \uqg = \mathop{\oplus}_{\lambda\in Q} U^+ K_\lambda S(U^-)
\end{align*}
where $S$ denotes the antipode of $\uqg$. For any $\lambda\in Q$ let 
\begin{align*}
  P_\lambda:\uqg\rightarrow  U^+ K_\lambda S(U^-)
\end{align*}
denote the projection with respect to this decomposition. Similarly, for any $\alpha, \beta \in Q^+$ let $\pi_{\alpha,\beta}:\uqg\rightarrow U^+_\alpha U^0 U^-_{-\beta}$ denote the projection with respect to the direct sum decomposition
\begin{align*}
  \uqg = \mathop{\oplus}_{\alpha, \beta\in Q^+} U^+_\alpha U^0 U^-_{-\beta}.
\end{align*}
To prove Theorem \ref{thm:Citaui} fix $i\in I\setminus X$ with $\tau(i)\neq i$ and set $j=\tau(i)$ and $m=1-a_{ij}$. To shorten notation define $\lambda_{ij}=m\alpha_i+\alpha_j$ and set $Q_{-\lambda_{ij}}=\id\ot P_{-\lambda_{ij}}\circ \pi_{0,0}$ as a vector space endomorphism of $\uqg\otimes \uqg$. By \cite[(7.8)]{a-Kolb12p} one has for $Y=F_{ij}(B_i,B_j)$ the relation
\begin{align}\label{eq:CijY}
  C_{ij}(\bc)=-(\id \ot \vep)\circ Q_{-\lambda_{ij}}(\kow(Y) - Y\ot K_{-\lambda_{ij}}).
\end{align}
To evaluate the right hand side of \eqref{eq:CijY} recall that the quantum Serre polynomial $Y$ is a linear combination of terms of the form $B_i^{m-k}B_jB_i^{k}$. Hence its coproduct can be computed using the formulas 
\begin{eqnarray*} 
\Delta(B_i)&=& B_i\otimes K_i^{-1}+1\otimes F_i +c_i\mathcal{Z}_i\otimes E_{j}K_i^{-1}+(\textrm{rest})_i \\
\Delta(B_j)&=& B_j\otimes K_j^{-1}+1\otimes F_j +c_j\mathcal{Z}_j\otimes E_{i}K_j^{-1}+(\textrm{rest})_j 
\end{eqnarray*} 
which follow from \eqref{eq:Bi-def} and \eqref{eq:W-def}. Any term of the product $\Delta(B_i)^{m-k}\Delta (B_j)\Delta (B_i)^{k}$ containing a factor $(\textrm{rest})_i$ or $(\textrm{rest})_j$ maps to zero under $\pi_{0,0}$. The term $( B_i\otimes K_i^{-1})^{m-k}( B_j\otimes K_j^{-1})( B_i\otimes K_i^{-1})^{k} $ cancels with the corresponding term in $Y\otimes K_{-\lambda_{ij}}$. The only remaining terms of $Q_{-\lambda_{ij}}(\kow(Y) - Y\ot K_{-\lambda_{ij}})$ that have weight zero in the second tensor factor are either the products of $m-1$ terms of the form $B_i\otimes K_i^{-1}$, one term of the form $1\otimes F_i $ and one term of the form $c_j\mathcal{Z}_j\otimes E_{i}K_j^{-1}$, or  $m-1$ terms of the form $B_i\otimes K_i^{-1}$, one term of the form $1\otimes F_j $ and one term of the form $c_i\mathcal{Z}_i\otimes E_{j}K_i^{-1}$. Because of the definition of $\pi_{0,0}$ as a projection from $U^+U^0U^-$ to $U^0$, terms in which $E_i$ appears before $F_i$ are sent to zero by $\pi_{0,0}$, as well as the terms in which $E_j$ appears before $F_j$. Finally, as $\mathcal{Z}_i$ $q$-commutes with $B_i,B_j$, we get that
\begin{align}\label{eq:ajai}
   Q_{-\lambda_{ij}}(\kow(Y) - Y\ot K_{-\lambda_{ij}}) = \big(a_j B_i^{m-1}c_j\cZ_j + a_iB_i^{m-1}c_i\cZ_i\big)\ot K_{-\lambda_{ij}}
\end{align}
for some $a_i,a_j\in \field(q)$. The coefficients $a_i$ and $a_j$ can be determined explicitly. Using $\cZ_j B_i = q^{(\alpha_j-\alpha_i,\alpha_i)} B_i \cZ_j=q_i^{-(m+1)}B_i\cZ_j$ one calculates
\begin{align*}
  a_j  B_i^{m-1}&c_j \cZ_j \ot K_{-\lambda_{ij}}= Q_{-\lambda_{ij}}\Big(\sum_{k=0}^m(-1)^k 
   \left[\begin{array}{c} m\\k\end{array} \right]_{q_i}\cdot\\
     &\qquad \cdot\sum_{l=0}^{m-k-1} B_i^l \, 1 \, B_i^{m-k-l-1} c_j \cZ_j B_i^{k} \ot K_i^{-l} F_i K_i^{-(m-k-l-1)}E_i K_j^{-1}K_i^{-k)}\Big)\\
    =&\sum_{k=0}^m\frac{(-1)^k}{q_i-q_i^{-1}} 
   \left[\begin{array}{c} m\\k\end{array} \right]_{q_i}\sum_{l=0}^{m-k-1}   
    q^{-(m+1)k-2(m-k-l-1)} B_i^{m-1}c_j \cZ_j \ot K_{-\lambda_{ij}}.
\end{align*}
This implies that 
\begin{align}
  a_j=&\sum_{k=0}^m\frac{(-1)^k}{q_i-q_i^{-1}} 
   \left[\begin{array}{c} m\\k\end{array} \right]_{q_i}\sum_{l=0}^{m-k-1} q_i^{-(m-1)k-2(m-1)} q_i^{2l}\nonumber\\
    =& \sum_{k=0}^m\frac{(-1)^k}{q_i-q_i^{-1}} 
   \left[\begin{array}{c} m\\k\end{array} \right]_{q_i} q_i^{- (m-1)k-2(m-1)} \frac{1-q_i^{2(m-k)}}{1-q_i^2}.   \label{eq:first-sum}
\end{align}
By \cite[0.2.(3)]{b-Jantzen96} one knows that 
\begin{align}\label{eq:q-binom-vanish}
  \sum_{k=0}^m(-1)^k 
   \left[\begin{array}{c} m\\k\end{array} \right]_{q} q^{(m-1)k}=0=\sum_{k=0}^m(-1)^k \left[\begin{array}{c} m\\k\end{array} \right]_{q} q^{-(m-1)k}
\end{align}
and hence
\begin{align}
 a_j=& \sum_{k=0}^m\frac{(-1)^k}{(q_i-q_i^{-1})^2} 
   \left[\begin{array}{c} m\\k\end{array} \right]_{q_i} q_i^{1-(m+1)k}.\label{eq:nearly-there}
\end{align}      
Recall the definition of the $q$-shifted factorial $(x;x)_n$ from \eqref{eq:phi-def} and observe that
\begin{align}\label{eq:binom-phi}
  \sum_{k=0}^m (-1)^k 
   \left[\begin{array}{c} m\\k\end{array} \right]_{q} q^{(m+1)k}=(q^2;q^2)_m \qquad \mbox{for all $m\in \N$}
\end{align}  
as is shown by induction on $m$. Inserting \eqref{eq:binom-phi} into \eqref{eq:nearly-there} one obtains
\begin{align}
  a_j = \frac{q_i}{(q_i-q_i^{-1})^2} (q_i^{-2};q_i^{-2}). \label{eq:aj}
\end{align}
Similarly, using $\cZ_i B_i = q_i^{m+1} B_i \cZ_i$ one calculates
\begin{align*}
  a_i  B_i^{m-1}&c_i \cZ_i \ot K_{-\lambda_{ij}}= Q_{-\lambda_{ij}}\Big(\sum_{k=0}^m(-1)^k 
   \left[\begin{array}{c} m\\k\end{array} \right]_{q_i}\cdot\\
     &\quad \qquad \cdot\sum_{l=0}^{k-1} B_i^{m-k}\, 1\, B_i^l c_i \cZ_i B_i^{k-l-1} \ot K_i^{-(m-k)} F_j K_i^{-l}E_j K_i^{-1}K_i^{-(k-l-1)}\Big)\\
     =&\sum_{k=0}^m\frac{(-1)^k}{q_i-q_i^{-1}} 
   \left[\begin{array}{c} m\\k\end{array} \right]_{q_i}\sum_{l=0}^{k-1} q_i^{(m+1)(k-l-1)+(m-1)l} B_i^{m-1}c_i \cZ_i \ot K_{-\lambda_{ij}}.
\end{align*}
This implies that
\begin{align}
  a_i \stackrel{\phantom{\eqref{eq:q-binom-vanish}}}{=}& \sum_{k=0}^m\frac{(-1)^k}{q_i-q_i^{-1}} 
   \left[\begin{array}{c} m\\k\end{array} \right]_{q_i}\sum_{l=0}^{k-1} q_i^{(m+1)(k-1)-2l} \nonumber\\
    \stackrel{\phantom{\eqref{eq:q-binom-vanish}}}{=}&\sum_{k=0}^m\frac{(-1)^k}{q_i-q_i^{-1}} 
   \left[\begin{array}{c} m\\k\end{array} \right]_{q_i}q_i^{(m+1)(k-1)}\frac{1-q_i^{-2k}}{1-q_i^{-2}} \nonumber \\
   \stackrel{\eqref{eq:q-binom-vanish}}{=}&\frac{ q_i^{-m} }{(q_i-q_i^{-1})^2}    
         \sum_{k=0}^m (-1)^k \left[\begin{array}{c} m\\k\end{array} \right]_{q_i} q_i^{(m+1)k} \nonumber\\
     \stackrel{\eqref{eq:binom-phi}}{=}& \frac{q_i^{-m} }{(q_i-q_i^{-1})^2} (q_i^2;q_i^2)_m. \label{eq:ai}        
\end{align}
Inserting the formulas \eqref{eq:aj} and \eqref{eq:ai} into \eqref{eq:ajai} and then into \eqref{eq:CijY} one obtains the statement of Theorem \ref{thm:Citaui}.
\providecommand{\bysame}{\leavevmode\hbox to3em{\hrulefill}\thinspace}
\providecommand{\MR}{\relax\ifhmode\unskip\space\fi MR }
\providecommand{\MRhref}[2]{%
  \href{http://www.ams.org/mathscinet-getitem?mr=#1}{#2}
}
\providecommand{\href}[2]{#2}


\begin{thebibliography}{BKLW14}

\bibitem[Ara62]{a-Araki62}
S.~Araki, \emph{On root systems and an infinitesimal classification of
  irreducible symmetric spaces}, J. Math. Osaka City Univ. \textbf{13} (1962),
  1--34.

\bibitem[BK15]{a-BalaKolb15p}
M.~Balagovi{\'c} and S.~Kolb, \emph{Universal {K}-matrix for quantum symmetric pairs}, preprint, {\ttfamily arXiv:1507.06276v1} (2015), 52 pp.

\bibitem[BKLW14]{a-BaoKuLiWang14p}
H.~Bao, J.~Kujawa, Y.~Li, and W.~Wang, \emph{Geometric {S}chur duality of
  classical type}, preprint, {\ttfamily arXiv:1404.4000v2} (2014), 42 pp.

\bibitem[BLM90]{a-BLM90}
A.~Beilinson, G.~Lusztig, and R~MacPherson, \emph{A geometric setting for the
  quantum deformation of ${G}{L}_n$}, Duke Math. J. \textbf{61} (1990),
  655--677.

\bibitem[BW13]{a-BaoWang13p}
H.~Bao and W.~Wang, \emph{A new approach to {K}azhdan-{L}usztig theory of type $B$ via quantum symmetric pairs}, preprint, {\ttfamily arXiv:1310.0103v1} (2013), 89 pp.

\bibitem[Cal93]{a-Caldero93}
P.~Caldero, \emph{\'{E}l\'ements ad-finis de certains groupes quantiques}, C.
  R. Acad. Sci. Paris (I) \textbf{316} (1993), 327--329.

\bibitem[Che84]{a-Cher84}
I.V. Cherednik, \emph{Factorizing particles on a half-line and root systems},
  Theoret. Math. Phys \textbf{61} (1984), 977--983.

\bibitem[ES13]{a-EhrigStroppel13p}
M.~Ehrig and C.~Stroppel, \emph{Nazarov-{W}enzl algebras, coideal subalgebras
  and categorified skew {H}owe duality}, preprint, {\ttfamily
  arXiv:1310.1972v2} (2013), 76 pp.

\bibitem[Jan96]{b-Jantzen96}
J.C. Jantzen, \emph{Lectures on quantum groups}, Grad. Stud. Math., vol.~6,
  Amer. Math. Soc, Providence, RI, 1996.

\bibitem[JL94]{a-JoLet2}
A.~Joseph and G.~Letzter, \emph{Separation of variables for quantized
  enveloping algebras}, Amer. J. Math. \textbf{116} (1994), 127--177.

\bibitem[Kac90]{b-Kac1}
V.~G. Kac, \emph{Infinite dimensional {L}ie algebras}, 3rd. ed., Cambridge
  University Press, Cambridge, 1990.

\bibitem[Kol14]{a-Kolb12p}
S.~Kolb, \emph{Quantum symmetric {K}ac-{M}oody pairs}, Adv. Math. \textbf{267} (2014), 395--469.

\bibitem[KS09]{a-KolbStok09}
S.~Kolb and J.~Stokman, \emph{Reflection equation algebras, coideal
  subalgebras, and their centres}, Selecta Math. (N.S.) \textbf{15} (2009),
  621--664.

\bibitem[KW92]{a-KW92}
V.G. Kac and S.P. Wang, \emph{On automorphisms of {K}ac-{M}oody algebras and
  groups}, Adv. Math. \textbf{92} (1992), 129--195.

\bibitem[Let99]{a-Letzter99a}
G.~Letzter, \emph{Symmetric pairs for quantized enveloping algebras}, J.
  Algebra \textbf{220} (1999), 729--767.

\bibitem[Let02]{MSRI-Letzter}
\bysame, \emph{Coideal subalgebras and quantum symmetric pairs}, New directions
  in {H}opf algebras (Cambridge), MSRI publications, vol.~43, Cambridge Univ.
  Press, 2002, pp.~117--166.

\bibitem[Let03]{a-Letzter03}
\bysame, \emph{Quantum symmetric pairs and their zonal spherical functions},
  Transformation Groups \textbf{8} (2003), 261--292.

\bibitem[Let04]{a-Letzter04}
\bysame, \emph{Quantum zonal spherical functions and {M}acdonald polynomials},
  Adv. Math. \textbf{189} (2004), 88--147.

\bibitem[Lev88]{a-Levstein88}
F.~Levstein, \emph{A classification of involutive automorphisms of an affine
  {K}ac-{M}oody {L}ie algebra}, J. Algebra \textbf{114} (1988), no.~2,
  489--518.

\bibitem[Lus94]{b-Lusztig94}
G.~Lusztig, \emph{Introduction to quantum groups}, Birkh{\"a}user, Boston,
  1994.

\bibitem[Skl88]{a-Sklyanin88}
E.~Sklyanin, \emph{Boundary conditions for integrable quantum systems}, J. Phys. A \textbf{21} (1988), 2375--2389.

\bibitem[Twi92]{a-Twietmeyer92}
E.~Twietmeyer, \emph{Real forms of ${U}_q(\mathfrak{g})$}, Lett. Math. Phys.
  \textbf{24} (1992), 49--58.

\end{thebibliography}
\end{document}